\documentclass[12pt,reqno]{amsart}

\usepackage{a4,url,amssymb,enumerate,tikz,scalefnt}

\usetikzlibrary{positioning}
\usetikzlibrary{arrows, arrows.meta,automata, trees}
\usetikzlibrary{backgrounds}
\usetikzlibrary{shapes,shapes.geometric,fit,calc,positioning,automata}

\usepackage[colorlinks=true,urlcolor=blue,linkcolor=blue,citecolor=blue]{hyperref}

\usepackage{pdfcomment}
\usepackage{amsmath}
\usepackage{amsfonts}
\usepackage{mathtools}
\usepackage{enumitem}
\usepackage{array}
\usepackage{booktabs,adjustbox}

\theoremstyle{plain}
\newtheorem{thm}{Theorem}[section]
\newtheorem{lemma}[thm]{Lemma}

\newtheorem{cor}[thm]{Corollary}
\newtheorem{claim}{Claim}

\theoremstyle{definition}

\newtheorem{example}[thm]{Example}
\newtheorem{examples}[thm]{Examples}

\theoremstyle{remark}

\newenvironment{thmenumerate}{
\begin{enumerate}[label=\textup{(\roman*)}, widest=(ii), leftmargin=9mm,itemsep=1mm,topsep=0mm]}{
\end{enumerate}}

\newenvironment{txtitemize}{
\begin{itemize}[leftmargin=6mm,itemsep=1mm,topsep=0mm]}{\end{itemize}}

\DeclareMathOperator{\Av}{Av}
\DeclareMathOperator{\Path}{Path}

\DeclareMathOperator{\pofw}{\Pi}
\DeclareMathOperator{\wofp}{\text{W}}
\DeclareMathOperator{\paofpe}{\Pi}
\DeclareMathOperator{\peofpa}{\Sigma}

\DeclareMathOperator{\perm}{perm}
\DeclareMathOperator{\Fact}{Fact}

\newcommand{\Bruijn}{\mathcal{B}}

\newcommand{\cl}{\mathfrak{c}}
\newcommand{\bl}{\mathfrak{b}}

\definecolor{lightgray}{rgb}{0.83, 0.83, 0.83}

\title[Consecutive orderings on words and permutations]{Atomicity and well quasi-order for consecutive orderings on words and permutations}

\author{M. McDevitt}
\author{N. Ru\v{s}kuc}
\address{School of Mathematics and Statistics, University of St Andrews, St Andrews, Scotland, UK}
\email{$\{$mm241,nr1$\}$@st-andrews.ac.uk}


\date{\today}

\begin{document}
\maketitle

\begin{abstract}
Algorithmic decidability is established for two order-theoretic properties of downward closed subsets defined by finitely many obstructions in two infinite posets.
The properties under consideration are: (a) being atomic, i.e. not being decomposable as a union of two downward closed proper subsets, or, equivalently, satisfying the joint embedding property;
and (b) being well quasi-ordered.
The two posets are: (1) words over a finite alphabet under the consecutive subword ordering; and (2) finite permutations under the consecutive subpermutation ordering.
Underpinning the four results are characterisations of atomicity and well quasi-order for the subpath ordering on paths of a finite directed graph.
 \end{abstract}

\section{Introduction}
\label{sec:intro}

The purpose of this paper is to prove the following results:

\begin{thm}
\label{thm:words}
It is algorithmically decidable, given a finite set of words $w_1,\dots,w_n$ over a finite alphabet $A$,
whether the set $\Av(w_1,\dots,w_n)$ of all words which avoid each of $w_1,\dots,w_n$ as a consecutive subword
is:
\begin{thmenumerate}
\item
\label{it:words-ato}
atomic;
\item
\label{it:words-wqo}
well quasi-ordered.
\end{thmenumerate}
\end{thm}

\begin{thm}
\label{thm:perms}
It is algorithmically decidable, given a finite set of permutations $\beta_1,\dots,\beta_n$,
whether the set $\Av(\beta_1,\dots,\beta_n)$ of all permutations which avoid each of $\beta_1,\dots,\beta_n$ as a consecutive subpermutation
is:
\begin{thmenumerate}
\item
\label{it:perms-ato}
atomic;
\item
\label{it:perms-wqo}
well quasi-ordered.
\end{thmenumerate}
\end{thm}

These results have been obtained during the first author's Ph.D. research, and appear in his dissertation \cite{mcdevitt-thesis}.
While the proofs in this paper are modelled on those from  \cite{mcdevitt-thesis}, some aspects are substantially different.
The well quasi-order result for words (Theorem \ref{thm:words} \ref{it:words-wqo}) had been proved earlier by Atminas, Lozin and Moshkov  \cite{atminas17} in greater generality, namely for all sets of the form $\Av(B)$ where $B$ is a regular set of obstructions.
We include it here for completeness, because it, and our new proof, fit neatly into the overall picture we present.

The two consecutive orderings featuring in the above theorems will be introduced in Sections \ref{sec:Bruijn} and \ref{sec:perms-ato}, respectively.
In what follows we introduce the notion of avoidance and the properties of atomicity and well quasi-order, which apply to arbitrary posets (and indeed quasi-ordered sets, but we will not require this extra level of generality).

Let $(A,\leq)$ be a poset.
For $B\subseteq A$, the \emph{avoidance set} of $B$ is
\[
\Av(B):=\{ a\in A\::\: b\nleq a \text{ for all } b\in B\},
\]
i.e. the set of all elements of $A$ which \emph{avoid} all elements of $B$.
This set is \emph{downward closed} in the sense that 
$x\in \Av(B)$ and $y\leq x$ imply $y\in \Av(B)$.
Conversely, every downward closed set $C$ has the form $\Av(B)$, e.g. for $B:=A\setminus C$.
Furthermore, if the set $A$ is \emph{well-founded}, meaning it does not contain infinite strictly decreasing sequences, 
then $B$ can also be chosen as the set of minimal elements of $A\setminus C$; this is an antichain, and is called the
 \emph{basis} for $C$. Note that the basis is unique by definition.
If the basis of $C$ is finite we say that $C$ is \emph{finitely based}.

Atomicity and well quasi-order are fundamental structural properties in the theory of posets.
The former arises from the decompositional approach, where one wants to express a downward closed  set as a union of two proper such subsets. The downward closed sets which cannot be decomposed in this way are called \emph{atomic}.
It is a folklore result (and can be proved by adapting \cite[Theorem 3.19]{vatterx}) that a downward closed set $C$ is atomic if and only if it satisfies the \emph{joint embedding property} (\emph{JEP}), or \emph{join property} for short: for all $x,y\in C$ there exists $z\in C$ such that
$x\leq z$ and $y\leq z$.

A poset $(A,\leq)$ is \emph{well quasi-ordered} (\emph{wqo} for short) if it is well-founded and has no infinite antichains.
This is a generalisation of the concept of well ordering for linear orders, and generally applies to quasi-orders (i.e. sets equipped with a reflexive, transitive binary relation). Even though in this paper we work exclusively with posets, we have decided to retain `quasi' in the name for compatibility with literature, although we note that some authors choose the name partial well-order, or the related adjective well-partially-ordered, e.g. \cite[Section 3]{vatterx}.

With the above terminology our main results can be phrased as saying that the properties of being atomic or being wqo are algorithmically decidable for finitely based downward closed sets of words and permutations under the contiguous subword/subpermutation orderings.

Both atomicity and wqo are concepts that have been studied extensively throughout combinatorics, and in particular for words and permutations, albeit under two different, but related orderings. Specifically, Higman's celebrated result \cite{higman52} asserts that the set $A^\ast$ of all words over a finite alphabet $A$ is wqo under the (non-contiguous, scattered) subword ordering. 
Thus all downward closed classes of $A^\ast$ are finitely based and wqo, and the property of being wqo is trivially decidable.
By way of contrast the atomicity problem for $A^\ast$ is open to the best of our knowledge.
As for permutations, the most natural and most extensively studied ordering is that of (non-contiguous) pattern involvement. Decidability of both atomicity and wqo under this ordering are major open problems in the area,
and are surveyed in considerable detail in \cite{jel17} and \cite{shnr15}, respectively.
In fact, wqo questions are a major topic throughout combinatorics, with the Robertson--Seymour Graph Minor theorem \cite{minor} being the most famous example.
A decidability point of view is espoused and surveyed in~\cite{che11}.

Consecutive patterns in permutations have attracted considerable interest over the years, mostly from the enumeration point of view.
We refer the reader to \cite{eli16} for an informative survey, and to \cite{dwyer18} for an example of recent developments.
In addition to the interest in their own right, we view the developments in this paper as a useful stepping stone towards the consideration of the atomicity and wqo problems for the general pattern involvement.
It is, however, worth mentioning that recently Braunfeld \cite{braunfeld:thesis,braunfeld} has proved some results which suggest that atomicity problem may well be undecidable for the general pattern involvement.
Specifically, he proved that the atomicity is undecidable for the class of finite graphs under the induced subgraph ordering and also for three-dimensional permutations, which are structures with three linear orders.

The proofs of our main results are to be found as follows: Theorem \ref{thm:words} in  Section \ref{sec:words-ato-wqo};
Theorem \ref{thm:perms} \ref{it:perms-ato} in Section \ref{sec:perms-ato}; and Theorem \ref{thm:perms} \ref{it:perms-wqo} in Section \ref{sec:perms-wqo}.
Each decidability result is actually a more or less immediate corollary of a result that gives an explicit characterisation of the property in question; these are Theorem \ref{thm:f-ato} for atomicity on words, 
Theorem \ref{thm:f-wqo} for wqo on words, Theorem \ref{thm:permatom} for atomicity on permutations, and Theorem \ref{thm:permwqo} for wqo on permutations.
The underlying general strategy in all proofs is the same, and relies on reducing the problem to the analogous problem for the set of all paths in a finite directed graph under the subpath ordering,
and Sections \ref{sec:graphs-ato}--\ref{sec:Bruijn} contain the necessary preparations. 
Specifically, in Section  \ref{sec:graphs-ato} we characterise atomicity for the subpath ordering, in Section \ref{sec:graphs-wqo} we do the same for wqo, and in Section \ref{sec:Bruijn} we describe the digraph associated with a finitely based downward closed set of words.

We conclude this section with a few notational conventions.
We use $\mathbb{N}$, $\mathbb{N}_0$, $\mathbb{Z}$ and $\mathbb{R}$ to denote the sets of natural numbers, non-negative integers, integers and real numbers, respectively.
We use the interval notation to mean intervals of integers.
The intervals with end-points $1$ and $n$ will be abbreviated to refer to $n$ only, e.g.
$[n]=[1,n]=\{1,2,\dots,n\}$, $[n)=[1,n)=\{1,\dots,n-1\}$, etc.

\section{Subpath ordering on directed graphs: atomicity}
\label{sec:graphs-ato}

A \emph{directed graph} (\emph{digraph}) is a set $G$ on which a binary relation $E=E(G)\subseteq G\times G$ is defined.
We often write $u\rightarrow v$ instead of $(u,v)\in E$.
A \emph{path} is a sequence $\pi$ of vertices $v_1,\dots, v_k$ such that $v_i\rightarrow v_{i+1}$, $i=1,\dots, k-1$;
we write $\pi=v_1\rightarrow v_2\rightarrow\dots\rightarrow v_k$.
The vertex $v_1$ is the \emph{start vertex} or \emph{initial vertex}; $v_k$ is the \emph{end vertex} or the \emph{terminal vertex}.
The number $k-1$ of edges the path traverses is its \emph{length}.
Note that we allow the \emph{empty path} $(v_1)$ starting and ending at $v_1$, with no edges and having length $0$; it is distinct from a \emph{loop} $v_1\rightarrow v_1$, which has length $1$.
Two paths $\pi,\tau$ can be \emph{concatenated}, provided the terminal vertex of $\pi$ equals the initial vertex of $\tau$.
The path $\pi$ is said to be \emph{simple}, if all its vertices are distinct.
It is a \emph{cycle} if $k\geq 2$ and $v_1=v_k$; and it is a \emph{simple cycle} if $v_1,\dots,v_{k-1}$ are distinct.
For any $1\leq i\leq j\leq k$, the path $v_i\rightarrow v_{i+1}\rightarrow\dots\rightarrow v_j$ is a \emph{subpath} of $\pi$.
This defines a partial order on the set $\Path(G)$ of all paths in $G$, which we call the \emph{subpath ordering} and denote $\leq_p$.

In this section we show that the atomicity problem is decidable for 
the collection of all posets of the form $\Path(G)$, where $G$ is an arbitrary finite digraph.
In order to state the requisite explicit characterisation of atomicity we need some more definitions.

Given a digraph $G$ we define a binary relation $\rightarrow^\ast$, where $u\rightarrow^\ast v$ means that there exists a path from $u$ to $v$.
This relation is reflexive and transitive, i.e. it is a quasi-order. 
The equivalence relation associated with the quasi-order $\rightarrow^\ast$
is defined by
$u\leftrightarrow^\ast v$ if and only if $u\rightarrow^\ast v$ and $v\rightarrow^\ast u$.
The equivalence classes of this relation are called the \emph{strongly connected components} of $G$.
If there is only one such component, $G$ is said to be \emph{strongly connected}.
The relation $\rightarrow^\ast$ induces a partial ordering on the set of all strongly connected components.

A \emph{bicycle} is a special digraph consisting of two vertex disjoint
simple cycles $\pi_1$ and $\pi_2$, connected by a simple directed path $\tau$ from $\pi_1$ to $\pi_2$,
such that no internal vertex of $\tau$ belongs to either $\pi_1$ or $\pi_2$; see Figure \ref{fig:bic}.
We use the terms \emph{initial cycle}, \emph{terminal cycle} and the \emph{connecting path} to refer to $\pi_1$, $\pi_2$ and $\tau$, respectively.
For convenience, we will allow either or both $\pi_1$ and/or $\pi_2$ to be absent,
and we will also allow a single cycle as a bicycle.
However, if both cycles are present, we require the connecting path to be non-trivial, i.e. the two cycles are not allowed to share a point.

\begin{figure}
\begin{center}

\begin{tikzpicture}

\foreach \x in {0,72,...,288}
   {\draw [-{Stealth [scale=1.4]}] (\x:10mm) arc (\x:\x+68:10mm);}

\foreach \x in {0,72,...,288}
   {\draw [fill] (\x:10mm) circle (0.7mm);}

\foreach \x in {0,60,...,300}
   {\draw [-{Stealth [scale=1.4]},shift={(50mm,0mm)}] (\x:10mm) arc (\x:\x+56:10mm);}

\foreach \x in {0,60,...,300}
   {\draw [fill,shift={(50mm,0mm)}] (\x:10mm) circle (0.7mm);}
   
\draw [fill] (20mm,0mm) circle (0.7mm);
\draw [fill] (30mm,0mm) circle (0.7mm);

\draw [-{Stealth [scale=1.4]}] (10mm,0mm)--(19mm,0mm);
\draw [-{Stealth [scale=1.4]}] (20mm,0mm)--(29mm,0mm);
\draw [-{Stealth [scale=1.4]}] (30mm,0mm)--(39mm,0mm);
 
\end{tikzpicture}

\caption{A bicycle.}
\label{fig:bic}
\end{center}

\end{figure}
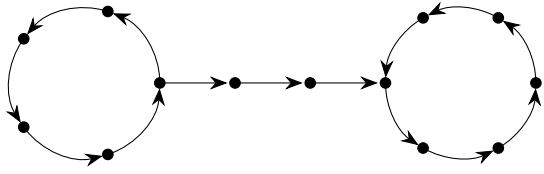

We are now ready to state and prove our atomicity result:

\begin{thm}
\label{thm:graphs-ato}
For a digraph $G$, the poset $\Path(G)$ of all directed paths in $G$ under the subpath ordering is
 atomic if and only if at least one of the following conditions holds:
\begin{thmenumerate}
\item
\label{it:grati}
$G$ is strongly connected; or
\item
\label{gratii}
$G$ is a bicycle.
\end{thmenumerate}
\end{thm}

As an immediate consequence we have:

\begin{cor}
\label{cor:graphs-ato}
It is algorithmically decidable, given a finite digraph $G$, whether or not the poset $\Path(G)$ of all directed paths in $G$ under the subpath ordering is
 atomic. \qed
\end{cor}

\begin{proof}[Proof of Theorem \ref{thm:graphs-ato}]
($\Leftarrow$)
First suppose that $G$ is strongly connected. Let $\pi$ and $\tau$ be two paths in $G$.
Let $u$ be the end vertex of $\pi$, let $v$ be the start vertex of $\tau$, and let $\xi$ be any path from
$u$ to $v$. Then the concatenation $\pi\xi\tau$ is a path containing both $\pi$ and $\tau$.

Now suppose that $G$ is a bicycle, with the initial cycle,  the connecting path and the final cycle
\begin{align*}
\pi_1 &= (v_1\rightarrow v_2\rightarrow \dots\rightarrow v_k\rightarrow v_1),
\\
\tau &= (v_1=u_1\rightarrow u_2\rightarrow \dots\rightarrow u_p= w_1),
\\
\pi_2 &= (w_1\rightarrow w_2\rightarrow \dots\rightarrow w_l\rightarrow w_1),
\end{align*}
respectively.
Atomicity follows from the fact that every path in $G$ is a subpath
of $\pi_1^m\tau\pi_2^n$ for sufficiently large $m,n$.
This argument remains valid if
$G$ is degenerate, i.e. if some of $\pi_1$, $\tau$ or $\pi_2$ are trivial.

($\Rightarrow$)
Suppose that $\Path(G)$ is atomic. Let $C_1,\dots,C_n$ be the strongly connected components of $G$. If $n=1$ then $G$ is strongly connected, and the proof is finished. So assume $n>1$.

\begin{claim}
\label{cl:gratcon}
For any two vertices there is a path from one of them to the other.
\end{claim}

\begin{proof}
Let $u,v$ be arbitrary distinct vertices, and let $\pi$ be a path that contains the trivial paths $(u)$ and $(v)$ as subpaths. If $\pi$ visits vertex $u$ before visiting vertex $v$ then it contains a subpath from $u$ to $v$, and otherwise it contains a subpath from $v$ to $u$.
\end{proof}

As a consequence of Claim \ref{cl:gratcon}, the components $C_1,\dots, C_n$ are linearly ordered, say $C_1\rightarrow^\ast C_2\rightarrow^\ast \dots\rightarrow^\ast C_n$.

\begin{claim}
\label{cl:gratind}
If $u$ is a vertex of in-degree $>1$ then $u\in C_n$.
\end{claim}

\begin{proof}
Let $v\rightarrow u$ and $w\rightarrow u$ be two distinct edges ending in $u$.
Suppose that $u\in C_k$ with $k<n$, and let $\pi$ be a path from $u$ to a vertex in $C_{k+1}$.
Then the paths $(v\rightarrow u)\pi$ and $(w\rightarrow u)\pi$ are not subpaths of any single path in $G$. 
\end{proof}

\begin{claim}
\label{cl:gratoutd}
If $u$ is a vertex of out-degree $>1$ then $u\in C_1$.
\end{claim}

\begin{proof}
This is dual to Claim \ref{cl:gratind}.
\end{proof}

From Claims \ref{cl:gratind}, \ref{cl:gratoutd} it immediately follows that the following hold:
\begin{txtitemize}
\item
Every vertex of $C_1$ has in-degree $0$ or $1$, and hence $C_1$ is trivial or a simple cycle.
\item
Every vertex of $C_n$ has out-degree $0$ or $1$, and hence $C_n$ is trivial or a simple cycle.
\item
Every vertex in $C_k$, where $k\in (n)$, has both the in-degree and the out-degree equal to $1$.
\end{txtitemize}
From this it follows that $G$ is a bicycle, if $n>2$. For the case $n=2$, notice that there is precisely one edge connecting $C_1$ and $C_2$, because two such edges cannot have a common path that contains them. Therefore $G$ is a bicycle when $n=2$ as well, and the proof is complete.
\end{proof}

\section{Subpath ordering on directed graphs: wqo}
\label{sec:graphs-wqo}

We now turn to the question of deciding the wqo property for the path poset $\Path(G)$ of a digraph $G$, and give two equivalent properties, one of which will be obviously algorithmically decidable.

We continue using the terminology and notation introduced in the previous section.
Additionally,
an \emph{in-out cycle} in a digraph $G$ is a cycle that contains a vertex of in-degree $>1$ and a vertex of out-degree $>1$.
A collection of subgraphs $H_1,\dots, H_n$ of a digraph $G$ will be called a \emph{path-complete decomposition} of $G$ if 
\[
\quad \Path(G)=\bigcup_{i=1}^n\Path(H_i),
\]
i.e. if every path in $G$ is wholly contained within one of the $H_i$.
Note that the $H_i$ must collectively contain all the vertices of $G$, but that they will not be disjoint in general.

\begin{thm}
\label{thm:grwqo}
The following are equivalent for a finite digraph $G$:
\begin{thmenumerate}
\item
\label{it:grwi}
$G$ has no in-out cycles.
\item
\label{it:grwii}
$G$ has a path-complete decomposition into bicycles.
\item
\label{it:grwiii}
The set $\Path(G)$ of all paths in $G$ is well quasi-ordered under the subpath ordering.
\end{thmenumerate}
\end{thm}

\begin{proof}
\ref{it:grwi}$\Rightarrow$\ref{it:grwii}
Suppose that $G$ has no in-out cycles. Then every strongly connected component is either trivial or a cycle.
Recall that the set of all strongly connected components of $G$ is partially ordered via the relation induced by 
$\rightarrow^\ast$. 
A component that is neither maximal nor minimal under this ordering certainly has an in-edge coming into it from another component, and an out-edge going into another component. Therefore such a component must be trivial.
It follows that every subgraph of $G$ consisting of a maximal component, a minimal component, and a single path connecting them, is a bicycle.
Taking all the subgraphs of this form, as well as all the components which are simultaneously maximal and minimal, yields a path-complete decomposition into bicycles.

\ref{it:grwii}$\Rightarrow$\ref{it:grwiii}
Suppose $G$ has a path-complete decomposition into bicycles.
Since $\Path(G)$ is a finite union of sets of the form $\Path(\mathcal{B})$, where $\mathcal{B}$ is a bicycle, 
it is sufficient to show that this latter set  is wqo, which amounts to showing that it has no infinite antichains.

So, consider a bicycle $\mathcal{B}$ consisting of the initial cycle,  the connecting path and the terminal cycle
\begin{align*}
\pi_1 &= (v_1\rightarrow v_2\rightarrow \dots\rightarrow v_k\rightarrow v_1),
\\
\tau &= (v_1=u_1\rightarrow u_2\rightarrow \dots\rightarrow u_p= w_1),
\\
\pi_2 &= (w_1\rightarrow w_2\rightarrow \dots\rightarrow w_l\rightarrow w_1),
\end{align*}
respectively.
A path in $\mathcal{B}$ is a concatenation of some of the following:
a proper subpath of $\pi_1$;
a concatenation $\pi_1^i$ for some $i\in\mathbb{N}$;
a proper subpath of $\pi_1$;
a subpath of $\tau$;
a proper subpath of $\pi_2$;
a concatenation $\pi_2^j$ for some $j\in\mathbb{N}$;
a proper subpath of $\pi_2$.
The set of all paths of any one of these types is clearly wqo.
A straightforward application of Higman's wqo result \cite{higman52} implies that $\Path(\mathcal{B})$ is wqo.

\ref{it:grwiii}$\Rightarrow$\ref{it:grwi}
We prove the contrapositive: if $G$ has an in-out cycle 
$\pi=u_1\rightarrow\dots\rightarrow u_k\rightarrow u_1$ then $\Path(G)$ is not wqo.
Let $u_i$ be a vertex of in-degree $>1$ with $v\rightarrow u_i$, $v\neq u_{i-1}$, where
$u_0$ is interpreted as $u_k$.
Analogously, let $u_j$ be a vertex of out-degree $>1$ with $u_j\rightarrow w$, $w\neq u_{j+1}$.
The set of all concatenations of the form
\[
(v\rightarrow u_i\rightarrow u_{i+1}\rightarrow \dots\rightarrow u_1) \pi^k
(u_1\rightarrow \dots\rightarrow u_j\rightarrow w)
\]
for $k=1,2,\dots$, is an antichain in $\Path(G)$.
\end{proof}

\section{Contiguous subword ordering and de Bruijn graphs}
\label{sec:Bruijn}

Let $A$ be a finite alphabet. Denote by $A^\ast$ the set of all words over $A$.
A word $u$ is a \emph{contiguous subword} (or a \emph{factor}) of a word $v$ if $v$ is a juxtaposition $v=yuz$ for some
$y,z\in A^\ast$. This defines a partial ordering on $A^\ast$ denoted by $\leq_f$.

For a set $W\subseteq A^\ast$ and $k\in \mathbb{N}_0$, we 
define
$W_k:=\{ w\in W\::\: |w|=k\}$, the set of all words in $W$ of length $k$.
For a set $K\subseteq \mathbb{N}_0$ let
$W_K:=\bigcup_{k\in K} W_k$.

For $m\in\mathbb{N}$, the \emph{$m$-dimensional de Bruijn graph over $A$}, denoted $\Bruijn_{A,m}$
is the digraph with vertices $A^\ast_m$ and edges $au\rightarrow ua^\prime$ with
$a,a^\prime\in A$, $u\in A^\ast_{m-1}$.
In other words, two length $m$ words are connected by a directed edge if the length $m-1$ suffix of the first word equals the length $m-1$ prefix of the second.
For more details on these graphs and historical references see \cite[Chapter 8]{vanlint92}.

There is an obvious correspondence between $A^\ast_{[m,\infty)}$ and paths in $\Bruijn_{A,m}$:
\begin{txtitemize}
\item
For a word $w=a_1\dots a_l$ with $l\geq m$, the corresponding path is:
\[
\pofw(w):= (a_1\dots a_m)\rightarrow (a_2\dots a_{m+1})\rightarrow\dots\rightarrow
(a_{l-m}\dots a_{l-1})\rightarrow (a_{l-m+1}\dots a_l).
\]
\item
For a path
\[
\pi = (a_1u_1\rightarrow u_1a_1^\prime=a_2u_2\rightarrow u_2a_2^\prime=a_3u_3\rightarrow\dots
\rightarrow u_{l-1}a_{l-1}^\prime=a_lu_l)
\]
the corresponding word is 
\[
\wofp(\pi):=a_1a_2\dots a_l u_l(=a_1u_1a_1^\prime\dots a_{l-1}^\prime).
\]
\end{txtitemize}
For us, the important properties of this correspondence are:
\begin{enumerate}[label=\textsf{(W$\Pi$\arabic*)}, widest=(W$\Pi$2), leftmargin=15mm]
\item
\label{it:wp1}
$\pofw$ and $\wofp$ are mutually inverse, i.e.
\[
\wofp(\pofw(w))=w \quad \text{and} \quad \pofw(\wofp(\pi))=\pi
\]
for all $w\in A^\ast_{[m,\infty)}$ and all $\pi\in\Path(\Bruijn_{A,m})$.
\item
\label{it:wp2}
$\pofw$ and $\wofp$ respect the factor ordering on $A^\ast_{[m,\infty)}$ and the subpath ordering on $\Path(\Bruijn_{A,m})$, in the sense that
\[
w_1\leq_f w_2\Rightarrow \pofw(w_1)\leq_p\pofw(w_2) \quad \text{and}\quad
\pi_1\leq_p \pi_2\Rightarrow \wofp(\pi_1)\leq_f\wofp(\pi_2).
\]
\end{enumerate}

Now consider a finitely based downward closed set $C\subseteq A^\ast$ under the factor ordering.
Let $B$ be the basis of $C$, and let $\bl:=\max\{|v|\::\: v\in B\}$, the maximum length of a basis element.
The \emph{factor graph} $\Gamma(C)$ is the induced subgraph of $\Bruijn_{A,\bl}$ on the vertex set
$C_\bl$.

\begin{examples}
\label{graph_example}
We conclude this section with three examples of downward closed sets over the alphabet $A=\{a,b\}$ and their factor graphs. In all the examples we have $C:=\Av(B)$ and $\bl:=\max\{|v|\::\: v\in B\}=3$.
The graphs are shown in Figure \ref{fig:fgs}. The bases and vertex sets are as follows:
\begin{thmenumerate}
\item
\label{it:fg1}
$B = \{bb, aab\}$,  $C_3 = \{bab, aba, baa, aaa\}$;
\item
\label{it:fg2}
$B = \{aaa, baa, bba, bbb\}$,  $C_3 = \{aab, abb,$ $ aba, bab\}$;
\item
\label{it:fg3}
$B = \{aa, aba, abb, bab\}$,  $C_3 = \{bbb, bba\}$.
\end{thmenumerate}
\end{examples}

\begin{figure}
\begin{center}
\begin{tabular}{cc}
$C$ & $\Gamma(C)$ \\ \hline\hline
\raisebox{5mm}{$\Av(bb, aab)$}
&
\begin{tikzpicture}[> =  {Stealth [scale=1.3]}, thick]
\tikzstyle{everystate} = [thick]
\node [state, shape = ellipse, minimum size = 20pt] (bab) at (0,0) {$bab$};
\node [state, shape = ellipse, minimum size = 20pt] (aba) at (2.5, 0) {$aba$};
\node [state, shape = ellipse, minimum size = 20pt]  (baa) at (5,0) {$baa$};
\node [state, shape = ellipse, minimum size = 20pt] (aaa) at (7.5, 0) {$aaa$};

\path[->]
(bab) edge [bend left] node {} (aba)
(aba) edge [bend left] node {} (bab)
(aba) edge node {} (baa)
(baa) edge node {} (aaa)
(aaa) edge [loop right] node {} (aaa)
;
\end{tikzpicture}
\\
&\\
\raisebox{18mm}{$\Av(aaa, baa, bba, bbb)$}
&
\begin{tikzpicture}[> =  {Stealth [scale=1.3]}, thick]
\tikzstyle{every state}=[thick]
\node [state, shape = ellipse, minimum size = 15pt] (0) at (0, 0) {$aab$};
\node [state, shape = ellipse, minimum size = 15pt]  (1) at (2, 0) {$aba$};
\node [state, shape = ellipse, minimum size = 15pt]  (2) at (4.5, 0) {$bab$};
\node [state, shape = ellipse, minimum size = 15pt]  (3) at (6.5, 0) {$abb$};

\path[->]
(0) edge node {} (1)
(1) edge [bend left] node {} (2)
(2) edge [bend left] node {} (1)
(2) edge node {} (3)
(0) edge [bend right = 45] node {} (3)
;
\end{tikzpicture}
\\&\\
\raisebox{2mm}{$\Av(aa, aba, abb, bab)$}
&
\begin{tikzpicture}[> =  {Stealth [scale=1.3]}, thick]
\tikzstyle{every state} = [thick]
\node [state, shape = ellipse, minimum size = 15pt]   (1) at (0,0) {$bbb$};
\node [state, shape = ellipse, minimum size = 15pt] (2) at (2.5, 0) {$bba$};
\path[->]
(1) edge [loop left, looseness = 10] node {} (1)
(1) edge node {} (2)
;
\end{tikzpicture}
\\ \hline\hline
\end{tabular}
\end{center}

\caption{Three downward closed sets under the factor ordering and their factor graphs.}
\label{fig:fgs}
\end{figure}
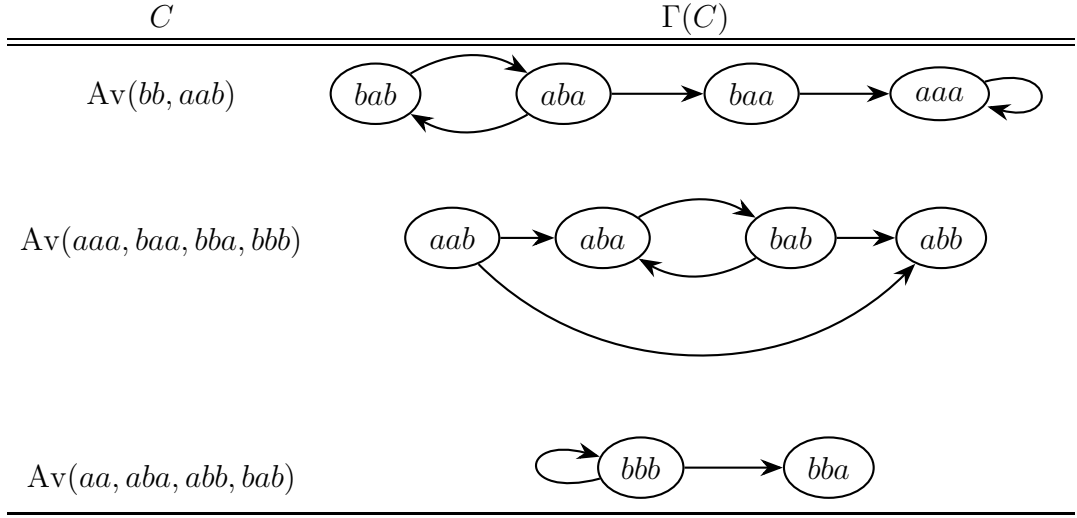

\section{Contiguous subword ordering: atomicity and wqo}
\label{sec:words-ato-wqo}

In Sections \ref{sec:graphs-ato}, \ref{sec:graphs-wqo} we have seen that the atomicity and wqo are decidable 
for path posets of finite digraphs.
In the previous section we have established a link between downward closed sets of words under the factor ordering and their factor graphs. We are now going to combine all these results and show that atomicity and wqo are decidable for factor ordering on words.

We begin with atomicity:

\begin{thm}
\label{thm:f-ato}
Let $A$ be a finite alphabet, let $C\subseteq A^\ast$ be an infinite, finitely based downward closed set under the factor ordering, and let $\bl$ be the maximum length of a basis element of $C$.
Then $C$ is atomic if and only if the following two conditions are satisfied:
\begin{thmenumerate}
\item
\label{it:fat1}
the factor graph $\Gamma(C)$ is strongly connected or is a bicycle; and
\item
\label{it:fat2}
for each word $u\in C$ with $|u|<\bl$ there is a word $v\in C$ with $|v|=\bl$ which contains $u$ as a factor.
\end{thmenumerate}
\end{thm}

\begin{proof}
($\Leftarrow$)
Suppose that \ref{it:fat1} and \ref{it:fat2} hold.
Let $u,v\in C$ be arbitrary.
Using \ref{it:fat2} if necessary, there exist $u^\prime,v^\prime\in C_{[\bl,\infty)}$ such that
$u\leq_f u^\prime$ and $v\leq_f v^\prime$.
From \ref{it:fat1} and Theorem \ref{thm:graphs-ato} we have that $\Path(\Gamma(C))$ is atomic;
hence $\pofw(u^\prime)$ and $\pofw(v^\prime)$ are subpaths of some path $\pi\in\Path(\Gamma(C))$.
Using \ref{it:wp1}, \ref{it:wp2} it follows that $u^\prime,v^\prime$ are factors of $\wofp(\pi)\in C$, and hence so are
$u,v$.

($\Rightarrow$)
We prove the contrapositive.
If \ref{it:fat1} is not satisfied then $\Path(\Gamma(C))$ is not atomic by Theorem \ref{thm:graphs-ato}.
Hence there exist two paths $\pi$ and $\tau$ which are not subpaths of a single path in $\Gamma(C)$.
But then the words $\wofp(\pi)$ and $\wofp(\tau)$ are not factors of a single word from $C$ because of properties
\ref{it:wp1}, \ref{it:wp2}, and therefore $C$ is not atomic.

Suppose now that \ref{it:fat2} does not hold, so there is $u\in C_{[\bl)}$ such that
$u\nleq_f v$ for all $v\in C_\bl$.
Since $C$ is infinite we must have $C_\bl\neq\emptyset$; let $w\in C_\bl$ be arbitrary.
Then $u$ and $w$ cannot be factors of some $z\in C$. Indeed, if they were, then $|z|\geq \bl$, and hence $z$ would contain a factor $v$ of length $\bl$, which in turn would contain $u$, a contradiction.
Hence $C$ is not atomic in this case either.
\end{proof}

The first part of  our first main theorem is an immediate corollary:

\begin{proof}[Proof of Theorem \ref{thm:words} \ref{it:words-ato}]
We need to verify that the conditions given in Theorem \ref{thm:f-ato} are algorithmically testable.
The factor graph $\Gamma(C)$ of $C=\Av(B)$ is computable from $B$.
Direct inspection of $\Gamma(C)$ establishes whether it is acyclic, which is the case precisely when $C$ is finite.
In this case, we compute all the elements of $C$ and directly check whether it satisfies the join property.
If $C$ is infinite, again by direct inspection
we determine
whether $\Gamma(C)$ is strongly connected or a bicycle, thus checking condition \ref{it:fat1}.
Finally \ref{it:fat2} can be checked by direct computation and inspection of all the elements of $C_{[\bl)}$.
\end{proof}

The link between the wqo property under the factor  and subpath orderings is even more direct.

\begin{thm}
\label{thm:f-wqo}
Let $A$ be a finite alphabet, and let $C\subseteq A^\ast$ be a finitely based downward closed set under the factor ordering.
Then $C$ is well quasi-ordered if and only if the factor graph $\Gamma(C)$ has a path-complete decomposition into bicycles.
\end{thm}

\begin{proof}
Let $C=\Av(B)$, where $B$ is finite, and let $\bl:=\max\{ |v|\::\: v\in B\}$.
Recall the properties \ref{it:wp1}, \ref{it:wp2} from Section \ref{sec:Bruijn}, which imply that $\pofw$ is a bijective, order-preserving correspondence between
$C_{[\bl,\infty)}$ and $\Path(\Gamma(C))$.
Hence $C_{[\bl,\infty)}$ is wqo if and only if $\Path(\Gamma(C))$ is wqo, which is the case if and only if 
$\Gamma(C)$ has a path-complete decomposition into bicycles by Theorem \ref{thm:grwqo}.
Furthermore, $C\setminus C_{[\bl,\infty)}$ is finite, and so $C$ is wqo if and only if $C_{[\bl,\infty)}$ is wqo,
completing the proof.
\end{proof}

\begin{proof}[Proof of Theorem \ref{thm:words} \ref{it:words-wqo}]
This follows from Theorem \ref{thm:f-wqo} and the observation that it is decidable whether $\Gamma(C)$ 
has a path-complete decomposition into bicycles; indeed one can equivalently check whether $\Gamma(C)$ 
has any in-out cycles by Theorem \ref{thm:grwqo}.
\end{proof}

As we mentioned in the Introduction, Theorem \ref{thm:words} \ref{it:words-wqo} has been proved in greater generality, for all regular avoidance sets, in \cite{atminas17}.

\begin{examples}
\label{exa:fato}
Let us return to the three sets from Example \ref{graph_example} and Figure \ref{fig:fgs}.
For the first one, $C=\Av(bb,aab)$, the factor graph is a bicycle. 
Furthermore, $C_{\{1,2\}}=\{a,b,aa, ab,ba\}$, and each of those words is a factor of a word of length $3$ in $C$,
e.g. $ baa$ or $ aba$.
Hence $C$ is atomic and wqo by Theorems \ref{thm:f-ato}, \ref{thm:f-wqo}.
For $C=\Av(aaa, baa, bba, bbb)$ the factor graph contains an in-out cycle $aba\rightarrow bab\rightarrow aba$,
and hence $C$ is neither atomic nor wqo. 
Finally, for $C=\Av( aa, aba, abb, bab)$, the factor graph is a bicycle,
but the word $ab\in C_{2}$ is not a factor of any word in $C_3$.
Hence $C$ is not atomic but is wqo.
\end{examples}

\section{Contiguous subpermutation ordering: atomicity}
\label{sec:perms-ato}

We now turn to permutations under the consecutive subpermutation ordering. In order to emphasise parallels and differences with the factor ordering on words, we will reuse much of the terminology and notation.
The key will again be a link with the subpath ordering on digraphs, via the notion of factor graph.
However, a major additional obstacle will be the fact that while a permutation will still determine a unique path in the graph, the converse will no longer be true, and a single path may correspond to several permutations. Characterising and recognising this ambiguity will be the key to deciding both atomicity and wqo.

A \emph{permutation} is a sequence $\sigma=a_1a_2\dots a_n$ of elements from $[n]$ in which every element appears precisely once. 
The set of all permutations will be denoted by $S$.
For a set $C\subseteq S$ and $k\in \mathbb{N}_0$, we 
define
$C_k:=\{ \sigma\in C\::\: |\sigma|=k\}$, the set of all permutations in $C$ of length $k$.
For a set $K\subseteq \mathbb{N}_0$ let
$C_K:=\bigcup_{k\in K} C_k$.

We will think of permutations as canonical representatives of sequences of distinct numbers.
For a sequence $\tau$ of length $n$ we use the notation $\tau|_i$ to denote its $i$th entry.
For a set $I=\{i_1<\dots<i_l\}\subseteq [n]$ we write $\tau|_I$ for the subsequence
$(\tau|_{i_1},\dots,\tau|_{i_l})$.
If $\tau$ is a finite sequence of distinct numbers, we will denote by $\perm(\tau)$ the unique permutation
$\sigma$ of the same length satisfying
\[
\tau|_i\leq \tau|_j \Leftrightarrow \sigma|_i\leq \sigma|_j \quad\text{for all } i,j\in \bigl[|\tau|\bigr].
\]
For example $\perm(3,-1,1/2,\sqrt{2})=4123$.

We say that a permutation $\tau$ of length $n$ contains a permutation $\sigma$ of length $m$
as a \emph{contiguous subpermutation} if 
there exists $i\in [n-m+1]$ such that $\sigma=\perm(\tau|_{[i,i+m-1]})$.
Alternatively, we say that $\sigma$ is a \emph{factor} of $\tau $ and write $\sigma\leq_f \tau$.
For example, we have $213\leq_f 23154$, because of the subsequence $315$, but
$123\nleq_f 23154$.
The relation $\leq_f$ is a partial ordering on $S$.
This ordering has been subject of considerable study, and \cite{eli16} contains an up to date survey.

We now define permutation analogues of de Bruijn graphs from Section \ref{sec:Bruijn}.
For $m\in\mathbb{N}$, the
\emph{$m$-dimensional permutation factor graph}
$\Gamma_m(S)$
has vertices $S_m$, and for two vertices $\sigma,\tau$ there is a directed edge $\sigma\rightarrow\tau$
if and only if 
$\perm(\sigma|_{[2,m]})=\perm(\tau|_{[1,m-1]})$.
These graphs have already been considered in \cite{avg08}.

One may be tempted to think that $\Gamma_m(S)$ is 
the induced subgraph of $\Bruijn_{[m],m}$ on the set $S_m$.
This, however, is not the case: the loop $123\rightarrow 123$ is present in $\Gamma_3(S)$ but not in $B_{[3],3}$.

Every permutation $\sigma$ of length $n\geq m$ defines a unique path in $\Gamma_m(S)$:
\[
\paofpe(\sigma):= \perm(\sigma|_{[1,m]})\rightarrow \perm(\sigma|_{[2,m+1]})\rightarrow\dots\rightarrow
\perm(\sigma|_{[n-m+1,n]}).
\]
However, this time the converse  correspondence, from paths to permutations, is trickier.
A path in the factor graph uniquely determines the $m$-factors of a corresponding permutation, but 
not how the entries that are farther than $m$ apart compare to each other.
For example, in the graph $\Gamma_3(S)$, each of the permutations
$2143$ and $3142$ traces the path $213\rightarrow 132$.
Hence, this time we need to assign a \emph{set} of permutations to a path $\pi$:
\[
\peofpa(\pi):=\{\sigma\in S\::\: \paofpe(\sigma)=\pi\}.
\]
Alternatively, if $\pi=\tau_1\rightarrow \tau_2\rightarrow\dots\rightarrow\tau_n$, then $\peofpa(\pi)$
consists of all permutations $\sigma$ of length $n+m-1$ satisfying
\[
\perm(\sigma|_{[i,i+m-1]})=\tau_i \quad\text{for all } i=1,\dots,n.
\]
When applying $\Sigma$ to a set of paths $P\subseteq \Path(\Gamma_m(S))$ we will abuse the notation a little and write
$\Sigma(P)$ for the set
$\bigcup\{ \Sigma(\pi)\::\: \pi\in P\}$.

\begin{example}
Consider the graph $\Gamma_3(S)$. 
The path associated with the permutation $34512$ is
\[
\paofpe(34512)=\bigl(\perm(345)\rightarrow\perm(451)\rightarrow\perm(512)\bigr)=(123\rightarrow 231\rightarrow 312).
\]
The set of permutations associated to this path is
\[
\peofpa(123\rightarrow 231\rightarrow 312)=\{13524,14523,23514,24513,34512\}.
\]
One way of seeing this is to start with $123$, and interpret the edge $123\rightarrow 231$ as an instruction to insert a last entry of value $<2$. This yields two permutations of length $4$, namely
$1342$ and $2341$. Then the edge $231\rightarrow 312$ means that in each of these two permutations we need to insert a new last entry whose value is between the two entries preceding it.
There are two ways of doing this for $1342$ and three for $2341$, yielding the above five permutations.
\end{example}

To parallel the properties \ref{it:wp1}, \ref{it:wp2} for words, we now have:

\begin{enumerate}[label=\textsf{($\Sigma\Pi$\arabic*)}, widest=($\Sigma\Pi$2), leftmargin=15mm]
\item
\label{it:sp1}
For all $\sigma\in S_{[m,\infty)}$ and all $\pi\in\Path(\Gamma_m(S))$ we have
\[
\sigma\in \peofpa(\paofpe(\sigma)) \quad \text{and} \quad 
\paofpe(\tau)=\pi \text{ for all } \tau\in\peofpa(\pi).
\]
\item
\label{it:sp2}
$\paofpe$ respects $\leq_f$:
\[
\sigma_1\leq_f \sigma_2\Rightarrow \paofpe(\sigma_1)\leq_p\paofpe(\sigma_2).
\]
\item
\label{it:sp3}
$\peofpa$ respects $\leq_p$ in the following sense:
If $\pi_1\leq_p \pi_2$ then for every
$\rho_2\in \peofpa(\pi_2)$ there exists $\rho_1\in\peofpa(\pi_1)$ such
that $\rho_1\leq_f\rho_2$; and for every $\sigma_1\in\peofpa(\pi_1)$ there exists
$\sigma_2\in\peofpa(\pi_2)$ such that $\sigma_1\leq_f\sigma_2$.
\end{enumerate}

The properties \ref{it:sp1}, \ref{it:sp2} follow from the definition and the fact that every permutation corresponds to a unique path. 
For the first statement in \ref{it:sp3}, if 
\[
\pi_1=(\tau_k\rightarrow \tau_{k+1}\rightarrow\dots\rightarrow \tau_l)\leq_p
(\tau_1\rightarrow \tau_{2}\rightarrow\dots\rightarrow \tau_n)=\pi_2,
\]
we can simply take
\[
\rho_1=\perm(\rho_2|_{[k,m+l-1]}).
\]
Finally, the second statement in \ref{it:sp3} follows directly from the following technical lemma with $t=1$:

\begin{lemma}
\label{lem:2sub2}
Suppose that a path $\pi$ in $\Gamma_m(S)$ can be expressed as a juxtaposition 
\[
\pi=\pi_1\tau_1\pi_2\tau_2\dots \pi_t\tau_t\pi_{t+1},
\]
where $t\geq 1$ and all $\pi_2,\dots, \pi_t$ have length at least $m$.
Let $\sigma_i\in\peofpa(\tau_i)$ for $i=1,\dots,t$.
Then there exists a permutation
\[
\rho=\rho_1\nu_1\rho_2\nu_2\dots\rho_t\nu_t\rho_{t+1}\in \peofpa(\pi),
\]
 where
\begin{gather}
\nonumber
\perm(\nu_i)=\sigma_i\ (i=1,\dots,t),\\
\label{eq:2s2}
|\rho_1|=|\pi_1|,\ |\rho_i|=|\pi_i|-m\ (i=2,\dots,t),\ |\rho_{t+1}|=|\pi_{t+1}|,\\
\nonumber
|\nu_i|=|\tau_i|+m\ (i=1,\dots,t).
\end{gather}
\end{lemma}

Intuitively, the lemma says that, in a permutation corresponding to a path, subpermutations corresponding to sufficiently spaced subpaths can be picked freely and independently of each other.

\begin{proof}[Proof of Lemma \ref{lem:2sub2}]
In several stages we build a sequence $\rho'=\rho_1''\nu_1'''\dots\rho_t''\nu_t'''\rho_{t+1}''$ such that $\perm(\rho')$ is the desired permutation. At each stage we will add some entries to the sequence constructed previously. 
It will be taken as read that these new entries are chosen to be distinct from any already in use.

To begin the construction we pick any sequence $\rho_1'$ such that $\paofpe(\perm(\rho_1'))=\pi_1$.
Note that $|\rho_1'|=|\pi_1|+m$, and write $\rho_1'=\rho_1''\nu_1'$, where $|\rho_1''|=|\pi_1|$ and $|\nu_1'|=m$.
The permutation $\perm(\nu_1')$ is the end vertex of the path $\pi_1$, and hence the start vertex of the path $\tau_1$.
Since $\sigma_1\in\peofpa(\tau_1)$, there exists a sequence $\nu_1''$ such that for $\nu_1''':= \nu_1'\nu_1''$
we have $\perm(\nu_1''')=\sigma_1$.

Now split $\nu_1'''=\nu_1^{\text{iv}}\nu_1^{\text{v}}$ with $|\nu_1^{\text{v}}|=m$, and note that $\perm(\nu_1^{\text{v}})$
is the end vertex of $\tau_1$, and hence the start vertex of $\pi_2$.
Therefore there exists a sequence $\rho_2'$ such that $\paofpe(\perm(\nu_1^{\text{v}}\rho_2'))=\pi_2$.
From $|\nu_1^{\text{v}}\rho_2'|=|\pi_2|+m$ and $|\nu_1^{\text{v}}|=m$ it follows that
$|\rho_2'|=|\pi_2|\geq m$.
We can therefore write $\rho_2'=\rho_2''\nu_2'$ where $|\nu_2'|=m$.
Now $\nu_2'$ can be extended to a sequence $\nu_2'''$ such that $\perm(\nu_2''')=\sigma_2$.
This sequence is split $\nu_2'''=\nu_2^{\text{iv}}\nu_2^{\text{v}}$, $|\nu_2^{\text{v}}|=m$,
and the process continues to construct sequences $\rho_3'',\nu_3''',\dots,\rho_{t+1}''$.

Letting $\rho':=\rho_1''\nu_1'''\dots \rho_t''\nu_t'''\rho_{t+1}''$, the desired permutation is $\rho:=\perm(\rho')$,
with the decomposition $\rho=\rho_1\nu_1\dots\rho_t\nu_t\rho_{t+1}$, where
$|\rho_i|=|\rho_i''|$ ($i=1,\dots,t+1$) and $|\nu_i|=|\nu_i'''|$ ($i=1,\dots,t$).
Properties \eqref{eq:2s2} follow easily from the corresponding properties of $\rho'$ as recorded during construction.
\end{proof}

For a finitely based set $C=\Av(B)$ of permutations, let $\bl=\max\{ |\beta|\::\: \beta\in B\}$, and let $m\geq \bl$.
The \emph{$m$-dimensional factor graph} $\Gamma_m(C)$ is the induced subgraph of $\Gamma_m(S)$ on
the vertex set $C_m$. 
Clearly $C_{[m,\infty)}=\peofpa(\Path(\Gamma_m(C)))$, i.e. the permutations of length $\geq m$ in $C$ are precisely 
the permutations corresponding to the paths in $\Gamma_m(C)$.
The graph $\Gamma_\bl(C)$ will be called the \emph{factor graph} of $C$ and denoted by $\Gamma(C)$.

\begin{example}
\label{exa:123321}
Let $C=\Av(123,321)$. Then $\bl=3$ and $C_3=\{132,231,213,312\}$.
The factor graph $\Gamma(C)$ is shown in Figure \ref{fig:123321}.
\end{example}
 
 \begin{figure}
\begin{center}
\begin{tikzpicture}[> = {Stealth [scale=1.3]}, thick,bend angle=20]
\tikzstyle{every state}=[thick]
\node[state, shape = ellipse, minimum size = 15pt] (A) at (0,0) {132};
\node[state, shape = ellipse, minimum size = 15pt] (B) at (3.2,0) {213};
\node[state, shape = ellipse, minimum size = 15pt] (C) at (3.2,3) {231};
\node[state, shape = ellipse, minimum size = 15pt] (D) at (0,3) {312};
\path[->]
(A) edge [bend left] node {} (B)
(B) edge [bend left] node {} (A)
(B) edge [bend left] node {} (C)
(C) edge [bend left] node {} (B)
(C) edge [bend left] node {} (D)
(D) edge [bend left] node {} (C)
(A) edge [bend left] node {} (D)
(D) edge [bend left] node {} (A)
;
\end{tikzpicture}

\caption{The graph $\Gamma(\Av(123,321))$.}
\label{fig:123321}

\end{center}
\end{figure}
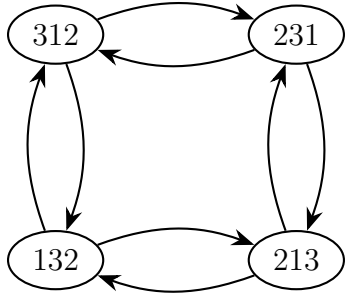

We call a path $\pi$ in $\Gamma(C)$ \emph{ambiguous} if $|\peofpa(\pi)|>1$ and \emph{unambiguous}
if $|\peofpa(\pi)|=1$.
We remark in passing that this notion has been  considered  in \cite{avg08}, albeit using a different terminology.
If $\pi$ is unambiguous, we will treat $\peofpa(\pi)$ as a permutation rather than a singleton set.
It turns out that, in addition to the conditions necessary for atomicity established for words in 
Theorem \ref{thm:f-ato}, in the permutation case the non-existence of ambiguous paths must also be added to the list.
We now proceed to establish this (Theorem \ref{thm:permatom}), and, having done that, we will be faced with the task of deciding the existence of ambiguous paths (Lemma \ref{lem:ambdec1}).

In all the lemmas in this section $C$ will denote a finitely based class of permutations under the factor ordering.
First we note that, as an immediate consequence of  Lemma \ref{lem:2sub2} for $t=1$, the set of unambiguous paths is downward closed:

\begin{lemma}
\label{lem:ambsub}
If $\pi$ is an unambiguous path in $\Gamma(C)$ and $\tau\leq_p\pi$ then $\tau$ is also unambiguous.
\end{lemma}

Next we prove that ambiguous paths are an obstacle to atomicity:

\begin{lemma}
\label{lem:ambxat}
If there is an ambiguous path $\pi$ that is not contained in a single strongly connected component of $\Gamma(C)$ then
$C$ is not atomic.
\end{lemma}

\begin{proof}
Suppose $\pi=\gamma_1\rightarrow \dots\rightarrow \gamma_l$, 
and let $ \sigma_1,\sigma_2$ be two distinct elements from $\peofpa(\pi)$.
We claim that there is no permutation $\rho\in C$ with $\sigma_1,\sigma_2\leq_f\rho$.
Suppose, aiming for contradiction, that there is.
The path $\paofpe(\rho)$ contains two distinct subpaths $\tau_1,\tau_2$ such that
 $\sigma_1\in\peofpa(\tau_1)$ and $\sigma_2\in\peofpa(\tau_2)$.
 This means that every edge $\gamma_i\rightarrow\gamma_{i+1}$ of $\pi$ is traversed twice by
 $\paofpe(\rho)$, and, as a consequence, all $\gamma_1,\dots,\gamma_l$ belong to the same strongly connected component of $\Gamma(C)$, a contradiction.
\end{proof}

\begin{lemma}
\label{lem:bcamb}
If $\Gamma(C)$ is a bicycle with more than one strongly connected component and if it has an ambiguous path
then $C$ is not atomic.
\end{lemma}

\begin{proof}
This is an immediate consequence of the previous lemma: either the ambiguous path visits at least two strongly connected components, or it can be extended to such a path, because the components in a bicycle are linearly ordered.
\end{proof}

\begin{thm}
\label{thm:permatom}
Let $C$ be an infinite, finitely based downward closed set of permutations under the consecutive subpermutation ordering, and let $\bl$ be the maximum length of a basis element of $C$.
Then $C$ is atomic if and only if the following two conditions are satisfied:
\begin{thmenumerate}
\item
\label{it:pat1}
the factor graph $\Gamma(C)$ is either strongly connected, or it is a bicycle with no ambiguous paths; and
\item
\label{it:pat2}
for each permutation $\sigma\in C$ with $|\sigma|<\bl$ there is a permutation $\rho\in C$ with $|\rho|=\bl$ which contains $\sigma$ as a contiguous subpermutation.
\end{thmenumerate}
\end{thm}

\begin{proof}
The proof is very similar to that of Theorem \ref{thm:f-ato}. The only substantive difference is the treatment of bicycles with ambiguous paths, which relies on Lemma \ref{lem:bcamb}.

($\Rightarrow$)
We prove the contrapositive.
Suppose first that \ref{it:pat1} does not hold. If $\Gamma(C)$ is a bicycle with an ambiguous path then $C$ is not atomic by Lemma \ref{lem:bcamb}.
Otherwise, $\Gamma(C)$ is neither strongly connected nor a bicycle, in which case $\Path(\Gamma(C))$ is not atomic by
Theorem \ref{thm:graphs-ato}.
Hence there exist two paths $\pi$ and $\tau$ which are not subpaths of a single path in $\Gamma(C)$.
Let $\sigma\in\peofpa(\pi)$, $\rho\in\peofpa(\tau)$ be arbitrary.
We claim that there does not exist $\nu\in C$ which contains both $\sigma$ and $\rho$ as contiguous subpermutations. Suppose it does.
By \ref{it:sp2}, the path $\paofpe(\nu)$ contains subpaths $\paofpe(\sigma)$ and $\paofpe(\rho)$, which equal 
$\pi$ and $\tau$ respectively by \ref{it:sp1}, a contradiction.

Suppose now that \ref{it:fat2} does not hold, so there is $\sigma\in C_{[\bl)}$ such that
$\sigma\nleq_f \rho$ for all $\rho\in C_\bl$.
Since $C$ is infinite we must have $C_\bl\neq\emptyset$; let $\nu\in C_\bl$ be arbitrary.
Then $\sigma$ and $\nu$ cannot be contiguous subpermutations of some $\mu\in C$. Indeed, if they were, then 
$|\mu|\geq |\nu|\geq \bl$, and 
the contiguous copy of $\sigma$ in $\mu$ would be contained in a contiguous subpermutation of length $\bl$
which would belong to $C$ by downward closure.
Hence $C$ is not atomic in this case either.

($\Leftarrow$)
Suppose that \ref{it:fat1} and \ref{it:fat2} hold.
Let $\rho,\sigma\in C$ be arbitrary.
Using \ref{it:fat2} if necessary, there exist $\rho^\prime,\sigma^\prime\in C_{[\bl,\infty)}$ such that
$\rho\leq_f \rho^\prime$ and $\sigma\leq_f \sigma^\prime$.
We claim that $\rho^\prime$ and $\sigma^\prime$ are factors of a single permutation from $C$.
It then follows that so are $\rho$ and $\sigma$, proving the joint embedding property, and hence $C$ is atomic as required.

To prove the claim, we let $\pi=\paofpe(\rho^\prime)$, $\tau=\paofpe(\sigma^\prime)$, 
noting that we need $|\rho^\prime|,|\sigma^\prime|\geq \bl$ for this step, 
and split the considerations into two cases, depending on which alternative from \ref{it:pat1} holds.

Suppose first that $\Gamma(C)$ is strongly connected.
Then there exists a path $\xi$ of length $\geq \bl$ from the end-point of $\pi$ to the start-point of $\tau$.
Applying Lemma \ref{lem:2sub2} (with $t=2$) to the path $\pi\xi\tau$ yields the existence of a permutation in $C$ which contains a prefix isomorphic to $\rho^\prime$ and a suffix isomorphic to $\sigma^\prime$.

Suppose now that $\Gamma(C)$ is a bicycle without ambiguous paths.
By Theorem \ref{thm:graphs-ato} we have that $\Path(\Gamma(C))$ is atomic;
hence $\pi$ and $\tau$ are subpaths of some path $\xi\in\Path(\Gamma(C))$.
Since $\Gamma(C)$ has no ambiguous paths, property \ref{it:sp3} implies that
the permutation $\peofpa(\xi)\in C$ contains the permutations $\peofpa(\pi)=\rho$ and $\peofpa(\tau)=\sigma$ as factors.
\end{proof}

We now turn towards proving that it is decidable whether $\Gamma(C)$ contains ambiguous paths.
We show how to check whether a single path is ambiguous (Lemma \ref{lem:ambdec}) and prove that the existence of ambiguous paths implies the existence of `short' ambiguous paths (Lemma \ref{lem:bound}).

The first lemma characterises when prepending an unambiguous path by an edge yields another unambiguous path.

\begin{lemma}
\label{lem:lins}
Let $\pi=\gamma_1\rightarrow\dots\rightarrow\gamma_m$ be a path in $\Gamma(C)$ such  that the subpath
$\gamma_2\rightarrow\dots\rightarrow\gamma_m$ is unambiguous.
Let $\sigma\in\peofpa(\pi)$ be arbitrary.
If $\gamma_1|_1\neq 1$, let $i$ be such that $\gamma_1|_i=\gamma_1|_1-1$; likewise
if $\gamma_1|_1\neq \bl$, let j be such that $\gamma_1|_j=\gamma_1|_1+1$.
Then
\[
\pi \text{ is unambiguous} \Leftrightarrow
\begin{cases}
\gamma_1|_1\not\in \{1,\bl\}\ \&\ \sigma|_j-\sigma|_i=2; \text{ or}\\
\gamma_1|_1=1\ \&\ \sigma|_j=2; \text{ or}\\
\gamma_1|_1=\bl\ \&\ \sigma|_i=\bl-1.
\end{cases}
\]
\end{lemma}

Informally, $\pi$ is unambiguous if and only if in $\sigma$ any neighbours by value of its initial entry are contained
in the prefix of length $\bl$.

\begin{proof}[Proof of Lemma \ref{lem:lins}]
We give a detailed proof for the generic case $1<\gamma_1|_1<\bl$.
For the `boundary cases' $\gamma_1|_1=1$ and $\gamma_1|_1=\bl$ the arguments are analogous.

($\Rightarrow$)
We prove the contrapositive. Suppose that $\sigma|_j-\sigma|_i>2$.
Then $\sigma|_j-\sigma|_1>1$ or $\sigma|_1-\sigma|_i>1$ (or both).
Suppose $\sigma|_j-\sigma|_1>1$; the other case is analogous.
Let $l$ be the unique subscript satisfying $\sigma|_l=\sigma|_1+1$.
Since $\gamma_1|_j-\gamma_1|_i=2$ we must have $l>\bl$; in other words, $\sigma|_l$ must occur outside the initial segment of $\sigma$ corresponding to $\gamma_1$.
Let $\rho$ be the permutation obtained from $\sigma$ by swapping $\sigma|_1$ and $\sigma|_l$.
By the preceding analysis the $\bl$-factors of $\rho$ and $\sigma$ coincide.
Hence $\paofpe(\rho)=\paofpe(\sigma)=\pi$, and hence $\pi$ is ambiguous.

($\Leftarrow$)
Suppose that $\sigma|_j-\sigma|_i=2$.
Let $\rho\in\peofpa(\pi)$ be arbitrary.
Since the path $\gamma_2\rightarrow\dots\rightarrow \gamma_m$ is unambiguous
we have
$\perm(\sigma|_{[2,m+\bl-1]})=\perm(\rho|_{[2,m+\bl-1]})$; let us call this permutation $\sigma^\prime$.
Since $\sigma^\prime$ is obtained from $\sigma$ by removing $\sigma|_1$ and subtracting $1$ from all values larger than $\sigma|_1$,
the entries $\sigma^\prime|_j$ and $\sigma^\prime|_i$ are adjacent in value.
Yet, in $\rho$ we have $\rho|_i<\rho|_1<\rho|_j$, as prescribed by the first factor $\gamma_1$.
It follows that $\rho$ is obtained from $\sigma^\prime$ by keeping all the entries of value $\leq \sigma^\prime|_i$ unchanged,
increasing by $1$ all the entries of value $\geq \sigma^\prime|_j$, and prepending a new initial entry $\rho|_1=\sigma^\prime|_i+1=\sigma|_1$.
In other words, we have $\rho=\sigma$, and hence $\pi$ is unambiguous.
\end{proof}

\begin{lemma}
\label{lem:ambdec}
It is decidable whether a given path in the graph $\Gamma(C)$ is ambiguous.
\end{lemma}

\begin{proof}
Let $\pi=\gamma_1\rightarrow\dots\rightarrow \gamma_m$.
We proceed recursively, backwards.
The trivial path $(\gamma_m)$ is unambiguous.
Then we use Lemma \ref{lem:lins} to determine whether the subpath $\gamma_{m-1}\rightarrow\gamma_m$ is ambiguous.
If it is, then $\pi$ is ambiguous by Lemma \ref{lem:ambsub}.
Otherwise, we again use Lemma \ref{lem:lins} on the subpath $\gamma_{m-2}\rightarrow\gamma_{m-1}\rightarrow\gamma_m$, and so on.
\end{proof}

\begin{lemma}
\label{lem:bound1}
Let $\pi=\gamma_1\rightarrow\dots\rightarrow\gamma_m$ be an unambiguous path in $\Gamma(C)$, and let 
$\sigma=\peofpa(\pi)$.
If $i\in [m+\bl-1]$ is such that $\sigma|_i-\sigma|_1=\pm 1$ then $i\leq \bl$.
\end{lemma}

\begin{proof}
This follows directly from Lemma \ref{lem:lins}.
\end{proof}

We are now in the position to establish a bound on the length of certain types of ambiguous paths:

\begin{lemma}
\label{lem:bound}
If $\pi=\gamma_1\rightarrow\dots\rightarrow\gamma_m$ is an ambiguous path in $\Gamma(C)$, 
such that the subpaths $\pi^\prime=\gamma_1\rightarrow\dots\rightarrow\gamma_{m-1}$
and $\pi^{\prime\prime}=\gamma_2\rightarrow\dots\rightarrow\gamma_m$ are unambiguous, then
$m\leq \bl$.
\end{lemma}

\begin{proof}
The lemma is vacuously true when $\bl=1$ or $m=1$: in the former case $\Gamma(C)$ is empty, and in the latter
the empty path is unambiguous. So suppose that $\bl,m\geq 2$, and note that this implies $|\sigma|\geq 3$.
Because of the (un)ambiguity assumptions the entries $\sigma|_1$ and $\sigma|_{m+\bl-1}$ must be adjacent in value.
Assume without loss of generality that $\sigma|_1=\sigma|_{m+\bl-1}-1$. 
Since $|\sigma|\geq 3$ we cannot have both $\sigma|_1=1$ and $\sigma|_{m+\bl-1}=m+\bl-1$;
without loss of generality assume $\sigma|_1>1$.
Let $j$ be such that $\sigma|_j=\sigma|_1-1$; we must have
$1<j<m+\bl-1$.
Since $\pi^\prime$ is unambiguous and
$\perm(\sigma|_{[m+\bl-2]})=\peofpa(\pi^\prime)$, the $j$th entry in $\perm(\sigma|_{[m+\bl-2]})$, i.e. the entry that corresponds to $\sigma|_j$,
must be within the first $\bl$ entries by Lemma \ref{lem:bound1}; in other words $j\leq \bl$.
By symmetry, considering the unambiguous path $\pi^{\prime\prime}$ and
the permutation $\perm(\sigma|_{[2,m+\bl-1]})$ we have $j\geq m$.
Combining the two inequalities yields $m\leq \bl$, completing the proof.
\end{proof}

Let us call an ambiguous path $\pi$ \emph{minimal} if all its proper subpaths are unambiguous.
Clearly, every ambiguous path contains a minimal ambiguous subpath. 
The previous lemma provides an effective bound on the length of ambiguous paths.
Since there are only finitely many paths of bounded length, and for each of them it can be checked whether it is ambiguous by Lemma \ref{lem:ambdec}, we conclude:

\begin{lemma}
\label{lem:ambdec1}
It is decidable whether $\Gamma(C)$ has ambiguous paths. \qed
\end{lemma}

\begin{proof}[Proof of Theorem \ref{thm:perms} \ref{it:perms-ato}]
The factor graph $\Gamma(C)$ is computable from $B$.
The set $C$ is finite if and only if $\Gamma(C)$ is acyclic, in which case we can enumerate all the elements of $C$ and check directly whether $C$ satisfies the join property.
So we suppose that $C$ is infinite and check whether $\Gamma(C)$ is strongly connected, in which case $C$ is atomic by Theorem \ref{thm:permatom}. 
Otherwise we check whether $\Gamma(C)$ is a bicycle. If it is not then $C$ is not atomic again by
Theorem \ref{thm:permatom}.
If $\Gamma(C)$ is a bicycle we check whether it contains any ambiguous paths; this is decidable by
Lemma \ref{lem:ambdec1}. 
If it does $C$ is not atomic, while if it does not $C$ is atomic, both by
Theorem \ref{thm:permatom}, completing the proof. 
\end{proof}

\begin{example}
The set $C=\Av(123,321)$ introduced in Example \ref{exa:123321} has the factor graph shown in Figure
\ref{fig:123321}. This graph is strongly connected and therefore $C$ is atomic.
\end{example}

\begin{example}
\label{ex:ttaa}
Let $B := \{231, 312, 321, 1243, 3142\}$ and $C := \Av(B)$.
We have $\bl=\max\{|\beta|\::\: \beta\in B\}= 4$ and $C_4 = \{1234, 1324, 2134, 2143\}$.
The
graph $G(C)$ is shown in Figure \ref{fig:g2143}; it is a bicycle.
By Lemma \ref{lem:bound}, any minimal ambiguous path in $G(C)$ has length at most $4$.
One can list all the paths of length $4$ and then check that they are all unambiguous, either by direct inspection or by applying the algorithm from the proof of Lemma \ref{lem:ambdec}.
Below we give all the paths and their corresponding permutations;
for ease of legibility we use $\alpha := 2143$, $\beta := 1324$, $\gamma := 2134$ and $\delta := 1234$:

\begin{center}
\begin{tabular}{ c | c }
$\pi$ & $\peofpa(\pi)$ \\
\hline
$\alpha\rightarrow\beta\rightarrow \alpha\rightarrow \beta$ & $2143657$ \\
$\alpha\rightarrow \beta\rightarrow \gamma\rightarrow \delta$ & $2143567$ \\
$\beta\rightarrow\alpha\rightarrow \beta\rightarrow \alpha$ & $1325476$ \\
$\beta\rightarrow \alpha\rightarrow \beta\rightarrow \gamma$ & $1324657$ \\
$\beta\rightarrow \gamma\rightarrow \delta\rightarrow \delta$ & $1324567$ \\
$\gamma\rightarrow \delta\rightarrow \delta\rightarrow \delta$ & $2134567$ \\
$\delta\rightarrow \delta\rightarrow \delta\rightarrow\delta$ & $1234567$ 
\end{tabular}
\end{center}
Furthermore, every permutation in the set $C_{<4} = \{1, 12, 21, 123, 213\}$ is a factor of $2134$.
Hence $C$ is atomic by Theorem \ref{thm:permatom}.
\end{example}

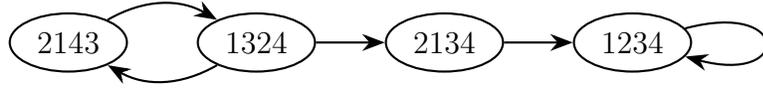
\begin{figure}
\begin{center}
\begin{tikzpicture}[> =  {Stealth [scale=1.3]}, thick]
\tikzstyle{every state}=[thick]
\node[state, shape = ellipse, minimum size = 15pt] (A) at (0,0) {2143};
\node[state, shape = ellipse, minimum size = 15pt] (B) at (2.5,0) {1324};
\node[state, shape = ellipse, minimum size = 15pt] (C) at (5,0) {2134};
\node[state, shape = ellipse, minimum size = 15pt] (D) at (7.5,0) {1234};
\path[->]
(A) edge [bend left] node {} (B)
(B) edge [bend left] node {} (A)
(B) edge node {} (C)
(C) edge node {} (D)
(D) edge [loop right, looseness = 10] node {} (D)
;
\end{tikzpicture}

\caption{The graph $\Gamma(\Av(231,312,321,1243,3142))$.}
\label{fig:g2143}

\end{center}
\end{figure}

\begin{example}
\label{exa:B321213x}
Let $B:=\{ 132,213,231,321\}$ and $C:=\Av(B)$.
This time we have $\bl=3$, $C_3=\{312,123\}$, and $\Gamma(C)$ is shown in Figure \ref{fig:123132etc}.
This graph is a bicycle, but the path $\pi:=312 \rightarrow 123$ is ambiguous, with
$\peofpa(\pi)=\{4123,3124\}$. Therefore $C$ is not atomic. The permutations $4123,3124$ do not join in $C$.
\end{example}

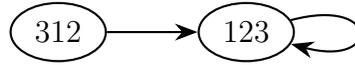
\begin{figure}
\begin{center}
\begin{tikzpicture}[> = {Stealth [scale=1.3]}, thick]
\tikzstyle{every state}=[thick]
\node[state, shape = ellipse, minimum size = 15pt] (1) at (0,0) {312};
\node[state, shape = ellipse, minimum size = 15pt] (2) at (2.5,0) {123};
\path[->]
(2) edge [loop right, looseness = 10] node {} (2)
(1) edge node {} (2)
;
\end{tikzpicture}

\caption{The graph $\Gamma(\Av(132,213,231,321))$.}
\label{fig:123132etc}
\end{center}

\end{figure}

\section{Contiguous subpermutation ordering: wqo}
\label{sec:perms-wqo}

In this section we turn to our final decidability result: wqo for permutations.
Following the pattern from the previous section,
the result will build on its counterpart for paths in graphs (Theorem \ref{thm:grwqo}), and will be similar to the corresponding result
for words (Theorem \ref{thm:f-wqo}), except for the effects of ambiguity.
However, this time these effects will be more subtle than in Section \ref{sec:perms-ato}.
It will turn out that, for a bicycle $\mathcal{B}$, ambiguous paths wholly contained within a cycle will be a definite obstacle to wqo, while other ambiguous paths will be such an obstacle only in certain circumstances.

The section is organised as follows:
\begin{txtitemize}
\item
First we show how to reduce the wqo question for an arbitrary finitely based class to the same question for the classes associated with bicycles with no ambiguous cycles; Lemma \ref{la:wqonec} and the discussion after it.
\item
Then we undertake a fine structural analysis of the two cycles in a bicycle; we classify  pairs of consecutive points in the associated permutations into increasing, decreasing, expanding and shrinking.
\item
Then we consider circumstances under which the rest of the bicycle allows points to be inserted into a shrinking pair; we prove that this constellation, called a splittable pair, represents yet another obstacle for wqo; Lemma \ref{la:splitac}.
\item
Finally, we present a technical argument which demonstrates 
that these conditions taken together are sufficient for a class associated with a bicycle to be wqo; Lemma \ref{la:wqob}.
\end{txtitemize}

Throughout $C$ will denote a fixed finitely based downward closed class of permutations, and $\bl$ will stand for the maximum length of a basis element for $C$.

So, as outlined above, we begin by presenting two conditions necessary for wqo:

\begin{lemma}
\label{la:wqonec}
If $C$ is wqo then the following hold:
\begin{thmenumerate}
\item
\label{it:wqonec1}
$\Gamma(C)$ has a path-complete decomposition into bicycles; and
\item
\label{it:wqonec2}
$\Gamma(C)$ has no ambiguous cycles.
\end{thmenumerate}
\end{lemma}

\begin{proof}
We prove the contrapositive.
Suppose first that \ref{it:wqonec1} does not hold. Then $\Path(\Gamma(C))$ is not wqo by Theorem \ref{thm:grwqo}.
Taking an infinite antichain $\{ \pi_i\::\: i=1,2,\dots\}$ in $\Gamma(C)$ and arbitrary
$\sigma_i\in \peofpa(\pi_i)$ yields an antichain in $C$ by \ref{it:sp2}.

Suppose now that \ref{it:wqonec2} does not hold. Let 
$\pi=\gamma_1\rightarrow\dots\rightarrow\gamma_t\rightarrow\gamma_1$ be an ambiguous cycle.
Let $l$ be the minimum length of an ambiguous subpath of $\pi$.

For every $i\in [t]$ consider the path $\pi_i=\gamma_i\rightarrow\dots\rightarrow\gamma_{i+l}$ with 
the convention that $\gamma_j=\gamma_{j-t}$ for $j>t$, and let $\delta_i\in\peofpa(\pi_i)$ be arbitrary.
Without loss of generality assume that $\pi_1$ is ambiguous. 

Note that for every $i$ we have
\[
\paofpe(\perm(\delta_i|_{[2,\bl+l]}))=\paofpe(\perm(\delta_{i+1}|_{[1,\bl+l-1]}))=
\gamma_{i+1}\rightarrow\dots\rightarrow\gamma_{i+l}.
\]
It follows that $\delta_i\rightarrow\delta_{i+1}$ is an edge in the $(\bl+l)$-dimensional factor graph $\Gamma_{\bl+l}(C)$, and so
$\xi:=\delta_1\rightarrow\dots\rightarrow\delta_t\rightarrow \delta_1$ is a cycle.

Since $\pi_1$ is ambiguous, there exists a permutation $\delta_1^\prime\in\peofpa(\pi_1)\setminus\{\delta_1\}$.
But then $\delta_1^\prime\rightarrow\delta_2$ and $\delta_t\rightarrow\delta_1^\prime$ are edges in
$\Gamma_{\bl+l}(C)$, and so $\xi$ is an in-out cycle.
It follows from Theorem \ref{thm:grwqo} that $\Path(\Gamma_{\bl+l}(C))$ is not wqo, and hence again $C$ itself is not wqo using \ref{it:sp2}.
\end{proof}

Suppose now that $\Gamma(C)$ has a path-complete decomposition into bicycles $\mathcal{B}_1,\dots,\mathcal{B}_n$.
Note that from
\[
\Path(\Gamma(C))=\bigcup_{i=1}^n\Path(\mathcal{B}_i)
\]
it follows that
\begin{equation}
\label{eq:fu}
C=C_{[\bl)}\cup \peofpa(\Path(\Gamma(C)))
=C_{[\bl)}\cup\bigcup_{i=1}^n \peofpa(\Path(\mathcal{B}_i)).
\end{equation}
Since this is a finite union, and the set $C_{[\bl)}$ is finite, we have that $C$ is wqo if and only if all 
$\peofpa(\Path(\mathcal{B}_i))$ are wqo.

So now we are dealing with the question: 
{\it for a fixed bicycle $\mathcal{B}=\mathcal{B}_i$, with no ambiguous cycles, when is 
$C_\mathcal{B}:=\peofpa(\Path(\mathcal{B}))$ wqo?}

To motivate the subsequent development, let us note that if $\mathcal{B}$ has no ambiguous paths whatsoever, then
$C_\mathcal{B}$ is certainly wqo.
Indeed, in the absence of ambiguity $\paofpe$ and $\peofpa$ are
mutually inverse, order preserving mappings between $C_\mathcal{B}$ and $\Path(\mathcal{B})$ by
\ref{it:sp1}--\ref{it:sp3}, and
$\Path(\mathcal{B})$ is wqo by Theorem \ref{thm:grwqo}.

On the other hand, it is possible for $\mathcal{B}$ to have ambiguous paths, while $C_\mathcal{B}$ remains wqo.
Indeed, for an easy example, take $\mathcal{B}$ to be any ambiguous path, and leave out the cycles.
Then $C_\mathcal{B}$ is finite, and hence wqo. This construction can easily be modified by the addition of an appropriate cycle at one or both ends of the path, to yield infinite wqo examples as well.
Finally, we note that it is possible for $\mathcal{B}$ to have no ambiguous cycles, and yet for $C_\mathcal{B}$ to be non-wqo, as the following example shows:

\begin{example}
\label{exa:B321213}
Let $\mathcal{B}$ the bicycle shown in Figure \ref{fig:123132etc} introduced in Example \ref{exa:B321213x}.
Even though $\mathcal{B}$ has ambiguous paths, they do not occur on the cycle.
On the other hand, $C_\mathcal{B}$ contains the infinite antichain
$\{(n-1,1,2,\dots,n-2,n)\::\: n\geq 4\}$.
\end{example}

We now proceed to undertake a more detailed structural analysis of a bicycle and the associated set of permutations.
We first fix some notation and terminology.

Our bicycle has an initial cycle, a connecting path, and a terminal cycle. One or both of the cycles may be trivial. However, in what follows we will assume that they are non-trivial.
The case where one of them is trivial is treated analogously.

Suppose the initial cycle is $\pi=v_1\rightarrow\dots\rightarrow v_r\rightarrow v_1$, where $v_1$ is the exit point onto the connecting path.
Let $s$ be the smallest such that $rs\geq \bl$. The concatenation
$\pi^s=v_1^\prime\rightarrow v_2^\prime\rightarrow\dots\rightarrow v_{rs}^\prime\rightarrow v_1^\prime$ is called the 
\emph{augmented initial cycle}, and the permutation
\[
\alpha:=\peofpa(v_1^\prime\rightarrow\dots\rightarrow v_{rs-\bl+1}^\prime)
\]
is called the \emph{initial iterator}.
It is uniquely determined because the cycles of $\mathcal{B}$ are unambiguous. 
To summarise: the initial iterator $\alpha$ is a uniquely determined permutation which repeatedly traverses the initial cycle,
and its length is both $\geq \bl$ and divisible by the length of the initial cycle.
The \emph{augmented terminal cycle} and the \emph{terminal iterator} are defined analogously, by considering the terminal cycle; the terminal iterator will be denoted by $\gamma$.

Now let
\begin{equation}
\label{eq:ccl}
\cl:=\max(|\alpha|,|\gamma|).
\end{equation}
We will focus on permutations $\sigma$ whose path $\paofpe(\sigma)$ starts at the beginning of the initial cycle, 
follows the augmented initial cycle a certain number $m\geq \cl$ times, then traverses the connecting path, enters the augmented terminal cycle, and winds around it again $m$ times.
Clearly every such permutation can be written as
\[
\sigma=\alpha_m\dots \alpha_1\beta\gamma_1\dots\gamma_m,
\]
where
\begin{txtitemize}
\item
$\perm(\alpha_m)=\dots=\perm(\alpha_1)=\alpha$;
\item
$\perm(\gamma_1)=\dots=\perm(\gamma_m)=\gamma$;
\item
$\paofpe(\perm(\beta))$ is the connecting path of $\mathcal{B}$, and $|\beta|$ is $\bl$ larger than the length of the connecting path.
\end{txtitemize}
We call these \emph{balanced} permutations. The number $m$ will be called the \emph{cycle length} of $\sigma$.
A balanced permutation of cycle length $\cl$
 is called a \emph{core}.
Every balanced permutation contains a core as a factor.
There are only finitely many cores, since they all have the same length $\cl|\alpha|+|\beta|+\cl|\gamma|$.

\begin{example}
\label{exa:bicrun}
Let $\mathcal{B}$ be the bicycle shown in Figure \ref{fig:bicex}.
We have $\bl =4$.
The initial cycle is $1234\rightarrow 1234$; it has length $1$, so $r=1$ and $s=4$.
The initial iterator $\alpha$ has length which is $\geq \bl=4$ and is a multiple of $1$; thus $|\alpha|=4$ and
$\alpha=1234$.
Similar considerations for the
the terminal cycle $1423\rightarrow 4132\rightarrow1423$ yield the 
terminal iterator $\gamma=1423$.
Thus we have $\cl=\max(|\alpha|,|\gamma|)=4$.
A core permutation is shown in Figure \ref{fig:bicex}, and is equal to
\begin{multline*}
(\underbrace{1,2,3,4}_{\alpha_4},\underbrace{5,6,7,8}_{\alpha_3},
\underbrace{9,10,11,12}_{\alpha_2},\underbrace{13,14,15,16}_{\alpha_1},
\\
\underbrace{17,18,19,30,20,40,21,39}_{\beta},\\
\underbrace{22,38,23,37}_{\gamma_1},
\underbrace{24,36,25,35}_{\gamma_2},
\underbrace{26,34,27,33}_{\gamma_3},
\underbrace{28,32,29,31}_{\gamma_4}.
\end{multline*}

Consider the 20th entry with value $30$ in this permutation.
The bicycle stipulates its relationship with the point preceding it, and hence with all the points preceding it, as they form an increasing sequence.
Its relationship with the succeeding three points is also fixed. However, the relationship with the points after that is not determined, and can be arbitrary.
Therefore, by varying the value of this point, we may obtain a further 17 core permutations.
\end{example}

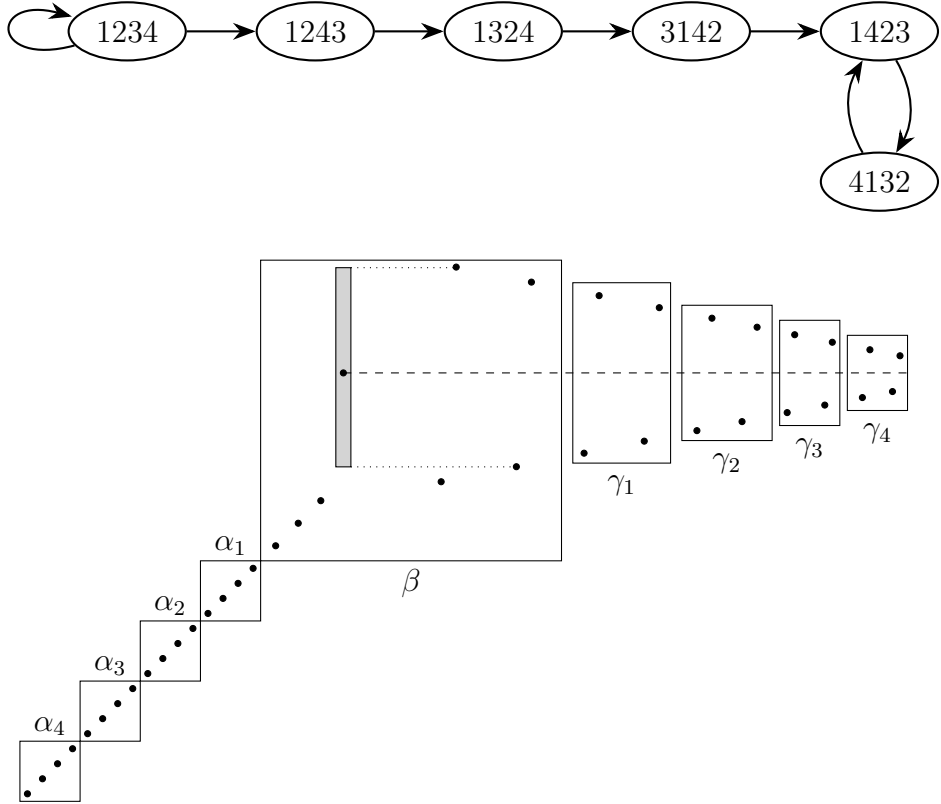
\begin{figure}
\begin{center}

\begin{tikzpicture}[> =  {Stealth [scale=1.3]}, thick]
\tikzstyle{every state}=[thick]
\node[state, shape = ellipse, minimum size = 15pt] (A) at (0,0) {1234};
\node[state, shape = ellipse, minimum size = 15pt] (B) at (2.5,0) {1243};
\node[state, shape = ellipse, minimum size = 15pt] (C) at (5,0) {1324};
\node[state, shape = ellipse, minimum size = 15pt] (D) at (7.5,0) {3142};
\node[state, shape = ellipse, minimum size = 15pt] (E) at (10,0) {1423};
\node[state, shape = ellipse, minimum size = 15pt] (F) at (10,-2) {4132};
\path[->]
(A) edge [loop left] node {}
(A) edge (B)
(B) edge (C)
(C) edge (D)
(D) edge (E)
(E) edge [bend left] (F)
(F) edge [bend left] (E)
;
\end{tikzpicture}
\\[6mm]
\begin{tikzpicture}

\draw [fill=lightgray]  (1,1.25 ) rectangle (1.2,3.9); 
\draw [dotted] (1.2, 1.25) -- (3.4,1.25);
\draw [dotted] (1.2, 3.9)--(2.6,3.9);
\draw (0,0) rectangle (4,4);
\node at (2,-0.3) {$\beta$};

\draw [fill] (0.2,0.2) circle (0.4mm);
\draw [fill] (0.5,0.5) circle (0.4mm);
\draw [fill] (0.8,0.8) circle (0.4mm);

\draw [fill] (1.1,2.5) circle (0.4mm);
\draw [dashed] (1.1,2.5)--(8.6,2.5);

\foreach \x in {2.4,3.4,4.3,5.1,5.8,6.4,7.0,7.5,8,8.4}
\draw [fill] (\x,0.2*\x+0.572) circle (0.4mm);

\foreach \x in {2.6,3.6,4.5,5.3,6,6.6,7.1,7.6,8.1,8.5}
\draw [fill] (\x,-0.2*\x+4.428) circle (0.4mm);

\draw (4.15,1.3) rectangle (5.45,3.7);
\node at (4.8,1) {$\gamma_1$};
\draw (5.6,1.6) rectangle (6.8,3.4);
\node at (6.2,1.3) {$\gamma_2$};
\draw (6.9,1.8) rectangle (7.7,3.2);
\node at (7.3,1.5) {$\gamma_3$};
\draw (7.8,2.0) rectangle (8.6,3.0);
\node at (8.2,1.7) {$\gamma_4$};

\draw (-0.8,-0.8) rectangle (0,0);
\node at (-0.4,0.2) {$\alpha_1$};
\draw (-1.6,-1.6) rectangle (-0.8,-0.8);
\node at (-1.2,-0.6) {$\alpha_2$};
\draw (-2.4,-2.4) rectangle (-1.6,-1.6);
\node at (-2,-1.4) {$\alpha_3$};
\draw (-3.2,-3.2) rectangle (-2.4,-2.4);
\node at (-2.8,-2.2) {$\alpha_4$};

\foreach \x in {0.1,0.3,0.5,0.7}
\foreach \y in {-0.8,-1.6,-2.4,-3.2}
{\draw [fill] (\x+\y,\x+\y) circle (0.4mm);}

\end{tikzpicture}

\caption{A bicycle and an associated core balanced permutation. Other core permutations are obtained by placing the indicated point elsewhere within the shaded rectangle.}

\label{fig:bicex}

\end{center}
\end{figure}

 A detailed structure analysis of balanced permutations will be the key technical step towards establishing the criterion for wqo.
The link with arbitrary permutations from $C_\mathcal{B}$ is via:

\begin{lemma}
Every permutation of $C_\mathcal{B}$ is a factor of a balanced permutation.
\end{lemma}

\begin{proof}
Let $\sigma\in C_\mathcal{B}$. Consider $\pi=\paofpe(\sigma)$. Extend $\pi$ to a balanced path and use \ref{it:sp3}.
\end{proof}

We begin with the following easy consequence of unambiguity of cycles, which asserts that the relative position of two entries in two different copies of the same iterator depends only on the distance between the two copies.

\begin{lemma}
\label{la:dist}
In any balanced permutation $\sigma=\alpha_m\dots \alpha_1\beta\gamma_1\dots\gamma_m$ we have:
\[
\gamma_s|_i\leq \gamma_t|_j\Leftrightarrow \gamma_1|_i\leq \gamma_{t-s+1}|_j
\quad \text{for all } i,j\in [|\gamma|],\ 1\leq s\leq t \leq m.
\]
\end{lemma}

\begin{proof}
If $s=t$ the assertion follows immediately from $\perm(\gamma_s)=\perm(\gamma_1)$. So suppose $s<t$.
Let $\rho_1$ be the factor of $\sigma$ corresponding to the subsequence starting at $\gamma_s|_i$ and ending at $\gamma_t|_j$,
and let $\rho_2$ be the
factor of $\sigma$ corresponding to the subsequence starting at $\gamma_1|_i$ and ending at $\gamma_{t-s+1}|_j$.
These two permutations trace the same path on the terminal cycle of $\mathcal{B}$. Since there are no ambiguous paths on the cycle by assumption, it follows that $\rho_1=\rho_2$, from which the assertion follows.
\end{proof}

\begin{lemma}
\label{la:twos}
Suppose $\sigma=\alpha_m\dots \alpha_1\beta\gamma_1\dots\gamma_m$ is a balanced permutation. Then the following hold:
\begin{thmenumerate}
\item
\label{la:twos1}
$\perm(\gamma_1\gamma_2)=\dots=\perm(\gamma_{m-1}\gamma_m)$.
\item
\label{la:twos2}
The sequence $(\gamma_1|_i,\gamma_2|_i,\dots,\gamma_m|_i)$ is monotone (increasing or decreasing) for every $i\in\bigl[|\gamma|\bigr]$.
\end{thmenumerate}
\end{lemma}

\begin{proof}
\ref{la:twos1} is a special case of Lemma \ref{la:dist},  and \ref{la:twos2} follows from \ref{la:twos1}. 
\end{proof}

We say that an entry of the terminal iterator $\gamma$ is \emph{increasing} or \emph{decreasing} depending on which alternative in
\ref{la:twos2} holds.
 Lemmas \ref{la:dist} and \ref{la:twos} remain valid when $\gamma$ is replaced by $\alpha$, 
 and so we can talk about increasing/decreasing entries of the initial iterator as well. 
 These words correspond to their intuitive meanings when  the sequence $\alpha_m\dots\alpha_1$ is read from right to left.

We now consider a pair of consecutive values $(\gamma|_i,\gamma|_j)$
of the terminal iterator $\gamma$,
where $\gamma|_j=\gamma|_i+1$, and discuss their possible relative positions in
the unique permutation $\gamma_1\gamma_2\in C_\mathcal{B}$ which is a juxtaposition of two copies of $\gamma$. 
We observe that of six different relative positions that the four values
$\gamma_1|_i $, $\gamma_1|_j$, $\gamma_2|_i$ and $\gamma_2|_j$ can theoretically assume, two are actually not possible:

\begin{lemma}
Let $\gamma_1\gamma_2\in C_\mathcal{B}$ be the unique permutation which is a juxtaposition of two copies of the terminal iterator $\gamma$.
Let $\gamma|_i$ and $\gamma|_j=\gamma|_i+1$ be a pair of consecutive values in $\gamma$. Then neither of the following two chains of inequalities is possible:
\[
\gamma_1|_i<\gamma_2|_i<\gamma_1|_j<\gamma_2|_j\quad \text{nor}\quad
\gamma_2|_i<\gamma_1|_i<\gamma_2|_j<\gamma_1|_j.
\]
\end{lemma}

\begin{proof}
Suppose $\gamma_1|_i<\gamma_2|_i<\gamma_1|_j<\gamma_2|_j$, the other chain being dual.
Consider the unique permutation $\rho=\gamma_1^\prime\gamma_2^\prime\gamma_3^\prime$
which is a juxtaposition of \emph{three} copies of $\gamma$.
Since 
\[
\perm(\gamma_1^\prime\gamma_2^\prime)=\perm(\gamma_2^\prime\gamma_3^\prime)=\perm(\gamma_1\gamma_2)
\]
we have
\[
\gamma_1^\prime|_i<\gamma_2^\prime|_i<\gamma_1^\prime|_j<\gamma_2^\prime|_j
\quad\text{and}\quad
\gamma_2^\prime|_i<\gamma_3^\prime|_i<\gamma_2^\prime|_j<\gamma_3^\prime|_j.
\]
Since for $t=1,2,3$, there are no values in $\gamma_t^\prime$ between $\gamma_t^\prime|_i$ and 
$\gamma_t^\prime|_j$, it follows that the values $\gamma_3^\prime|_i$ and $\gamma_1^\prime|_j$ are consecutive in $\rho$.
But they are also more than $\bl$ apart in position because $|\gamma_2|=|\gamma|=rs\geq \bl$. 
Hence the permutation $\rho^\prime$ obtained from $\rho$ by swapping these two values satisfies $\paofpe(\rho^\prime)=\paofpe(\rho)$, and hence $\rho^\prime\in C_\mathcal{B}$.
This means that the concatenation of three augmented terminal cycles is ambiguous, a contradiction.
\end{proof}

So we are left with the following four possibilities for the consecutive pair $(\gamma|_i,\gamma|_j)$:
\begin{txtitemize}
\item
\emph{increasing}: $\gamma_1|_i<\gamma_1|_j<\gamma_2|_i<\gamma_2|_j$; this is the case if and only if both $\gamma|_i$ and $\gamma|_j$ are increasing;
\item
\emph{decreasing}: $\gamma_2|_i<\gamma_2|_j<\gamma_1|_i<\gamma_1|_j$, if and only if $\gamma|_i$ and $\gamma|_j$ are decreasing;
\item
\emph{expanding}: $\gamma_2|_i<\gamma_1|_i<\gamma_1|_j<\gamma_2|_j$, if and only if $\gamma|_i$ is decreasing and $\gamma|_j$ is increasing;
\item
\emph{shrinking}: $\gamma_1|_i<\gamma_2|_i<\gamma_2|_j<\gamma_1|_j$, if and only if $\gamma|_i$ is increasing and $\gamma|_j$ is decreasing.
\end{txtitemize}

We extend this to the lower and upper edges of the permutation, where we talk about the pairs 
$(-\infty,1)$ and $(|\gamma|,\infty)$.
The pair $(-\infty,1)$ is \emph{expanding} if $\gamma_1|_1<\gamma_2|_1$,
\emph{shrinking} if $\gamma_2|_1<\gamma_1|_1$, and cannot be increasing or decreasing.
The definitions for $(|\gamma|,\infty)$ are analogous.
Again, the same analysis applies to the initial iterator as well.

\begin{example}
Let us return to the bicycle from Example \ref{exa:bicrun}, shown in Figure \ref{fig:bicex}.
The terminal iterator $\gamma=1423$ has five consecutive pairs $(-\infty,\gamma|_1)$, $(\gamma|_1,\gamma|_3)$,
$(\gamma|_3,\gamma|_4)$, $(\gamma|_4,\gamma|_2)$ and $(\gamma|_2,\infty)$.
Consider the unique juxtaposition of two copies of $\gamma$, namely
$18273645=\gamma_1\gamma_2$, where $\gamma_1=1827$, $\gamma_2=3645$.
From $\gamma_1|_1=1<\gamma_2|_1=3$, we conclude that $(-\infty,\gamma|_1)$ is an expanding pair.
Similarly, 
\[
\gamma_1|_1=1<2=\gamma_1|_3<3=\gamma_2|_1<4=\gamma_2|_3,
\]
and so $(\gamma|_1,\gamma|_3)$ is increasing.
The three remaining consecutive pairs
$(\gamma|_3,\gamma|_4)$, $(\gamma|_4,\gamma|_2)$ and $(\gamma|_2,\infty)$
are shrinking, decreasing and expanding, respectively.
The initial iterator $\alpha=1234$ also has five consecutive pairs;
the proper ones are all decreasing, while the boundary ones are expanding and shrinking respectively.
\end{example}

We can now use the above set-up to prove the following result which, combined with Lemma \ref{la:dist}, shows that the relative position of two points belonging to copies of the same iterator is completely determined by the core:

\begin{lemma}
\label{la:itcomp}
In any balanced permutation $\sigma=\alpha_m\dots \alpha_1\beta\gamma_1\dots\gamma_m$ we have:
\[
\gamma_1|_i\leq \gamma_s|_j\Leftrightarrow \gamma_1|_i\leq \gamma_\cl|_j\quad
\text{for all } i,j\in [|\gamma|],\  s\in (\cl,m].
\]
\end{lemma}

\begin{proof}
If $i=j$ the statement follows from Lemma \ref{la:twos} \ref{la:twos2}, so we may suppose without loss of generality
that $\gamma|_i<\gamma|_j$.
Let the entries of $\gamma$ with values in the range $[\gamma|_i,\gamma|_j]$ be
$\gamma|_i=\gamma|_{t_1}<\gamma|_{t_2}<\dots<\gamma|_{t_r}=\gamma|_j$.
If at least one of the entries $\gamma|_{t_l}$ is increasing then we have
\[
\gamma_1|_i\leq \gamma_1|_{t_l}< \gamma_s |_{t_l} \leq \gamma_s |_j
\quad\text{and}\quad
\gamma_1|_i\leq \gamma_1|_{t_l}< \gamma_\cl |_{t_l} \leq \gamma_\cl |_j.
\]
Otherwise, all these entries are decreasing, and hence so are the $r-1$ consecutive pairs corresponding to them.
Bearing in mind that $r\leq |\gamma|\leq \cl<s$ we now have
\[
\gamma_1|_i=\gamma_1|_{t_1}>\gamma_2|_{t_2}>\dots >\gamma_r|_{t_r}=\gamma_r|_j
\geq \gamma_\cl |_j > \gamma_s|_j.
\]
This completes the proof of the lemma.
\end{proof}

\begin{example}
\label{exa:B321213cont}
Consider the bicycle $\mathcal{B}$ from Example \ref{exa:B321213}, shown in Figure \ref{fig:123132etc}.
Its terminal iterator is $\gamma=123$; its consecutive pairs
$(-\infty,\gamma|_1)$, $(\gamma|_1,\gamma|_2)$, $(\gamma|_2,\gamma|_3)$, $(\gamma|_3,\infty)$
are expanding, increasing, increasing and shrinking, respectively.
The infinite antichain exhibited in Example \ref{exa:B321213} can be viewed as 
repeatedly traversing the terminal cycle and adding an initial
point, so that it splits
the shrinking pair in a particular manner.
\end{example}

Formalising from Example \ref{exa:B321213cont}, we say that a shrinking pair $(\gamma|_i,\gamma|_j)$ is
\emph{splittable} if there exists a core $\kappa=\alpha_\cl\dots\alpha_1\beta\gamma_1\dots\gamma_\cl$
and an entry $p$ in $\alpha_\cl\dots\alpha_1\beta$
such that $\gamma_1|_i<p<\gamma_1|_j$.
This includes the possibility that one or both of the boundary pairs $(-\infty,1)$, $(|\gamma|,\infty)$ may be splittable.
A splittable pair of $\alpha$ is defined analogously; the inserted point needs to come from
$\beta\gamma_1\dots\gamma_\cl$.

\begin{lemma}
\label{la:decsplit}
It is decidable whether a given bicycle has splittable  pairs.
\end{lemma}

\begin{proof}
Splittable pairs are manifested as such within cores, of which there are only finitely many, and which can be algorithmically constructed.
\end{proof}

\begin{lemma}
\label{la:splitac}
If $\mathcal{B}$ has a splittable pair, then $C_\mathcal{B}$ contains an antichain.
\end{lemma}

\begin{proof}
Without loss of generality suppose that 
there is a splittable pair $(\gamma|_i,\gamma|_j)$ within
the terminal iterator $\gamma$.
Furthermore, also without loss, we may assume that this terminal iterator is not $(|\gamma|,\infty)$.
Consider a core $\kappa=\alpha_\cl \dots \alpha_1\beta\gamma_1\dots\gamma_\cl $ in 
which there exists a point $p$ of $\alpha_\cl \dots \alpha_1\beta$ such that
$\gamma_1|_i<p<\gamma_1|_j$.
Take $p$ to be the rightmost such point, and consider the subsequence of $\kappa$ which begins at $p$ and ends with the first copy of the iterator.
This subsequence has the form $\sigma_1=\delta\gamma_1=p\delta_1\gamma_1$.
Successively extend $\sigma_1$ by further copies of the iterator on the right to obtain an infinite sequence of sequences $\sigma_m=\delta\gamma_1\dots\gamma_m$, $m=1,2,\dots$.
Each $\sigma_m$ starts with $p$, it is an initial segment of $\sigma_{m+1}$,
and $\perm(\sigma_m)\in C_\mathcal{B}$.

Consider an arbitrary $\sigma_m$.
By the rightmost choice of $p$, there are no other points in $\delta$ with value between
$\gamma_1|_i$ and $\gamma_1|_j$.
Since $\gamma|_i$, $\gamma|_j$ are consecutive in $\gamma$, there are no points in
any $\gamma_t$ between
$\gamma_t|_i$ and $\gamma_t|_j$.

Now we modify $\sigma_m$ to obtain two sequences $\sigma_m^\prime$, 
$\sigma_m^{\prime\prime}$, and a permutation $\sigma_m^{\prime\prime\prime}$, as follows:
\begin{txtitemize}
 \item
 $\sigma_m^\prime$ is obtained from $\sigma_m$ by removing the initial entry $p$;
\item
$\sigma_m^{\prime\prime}$ is obtained by inserting a new initial entry $p_m^{\prime}$ into 
$\sigma_m^\prime$, where $\gamma_m|_j<p_m^{\prime}<\gamma_{m-1}|_j$; 
this can be done because $\gamma|_j\neq\infty$ and is illustrated in Figure \ref{fig:insert};
\item
$\sigma_m^{\prime\prime\prime}=\perm(\sigma_m^{\prime\prime})$.
\end{txtitemize}
From the preceding analysis we have $\paofpe(\sigma_m^{\prime\prime\prime})=\paofpe(\perm(\sigma_m))$ and hence
$\sigma_m^{\prime\prime\prime}\in C_\mathcal{B}$.

We claim that $\{\sigma_m^{\prime\prime\prime}\::\: m=1,2,\dots\}$ is an antichain.
Indeed, suppose that for some $m<n$ we have
$\sigma_m^{\prime\prime\prime}\leq_f\sigma_n^{\prime\prime\prime}$.
Since both paths $\paofpe(\sigma_m^{\prime\prime\prime})$ and $\paofpe(\sigma_n^{\prime\prime\prime})$
traverse $\paofpe(\perm(\delta))$ before entering the terminal cycle, it follows that
$\sigma_m^{\prime\prime\prime}$ embeds into $\sigma_n^{\prime\prime\prime}$ as a prefix.
However, this is not possible, because in $\sigma_m^{\prime\prime\prime}$ we have
$\gamma_m|_j<p_m^{\prime}$, whereas in
$\sigma_n^{\prime\prime\prime}$ we have
$p_n^{\prime}<\gamma_{n-1}|_j\leq\gamma_m|_j$.
This proves the claim and with it the lemma.
\end{proof}

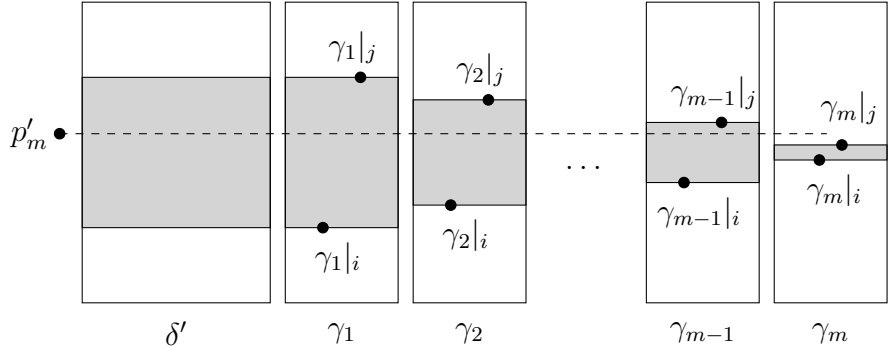
\begin{figure}
\begin{center}
\begin{tikzpicture}

\draw [fill=lightgray] (1.5,1) rectangle (4,3);
\draw  (1.5,0) rectangle (4,4);
\node at (2.75,-0.4) {$\delta^\prime$};

\draw [fill=lightgray] (4.2,1) rectangle (5.7,3);
\draw (4.2,0) rectangle (5.7,4);
\node at (4.95,-0.4) {$\gamma_1$};

\draw [fill=lightgray] (5.9,1.3) rectangle (7.4,2.7);
\draw (5.9,0) rectangle (7.4,4);
\node at (6.65,-0.4) {$\gamma_2$};

\draw [fill=lightgray] (9,1.6) rectangle (10.5,2.4);
\draw (9,0) rectangle (10.5,4);
\node at (9.75,-0.4) {$\gamma_{m-1}$};

\draw [fill=lightgray] (10.7,1.9) rectangle (12.2,2.1);
\draw (10.7,0) rectangle (12.2,4);
\node at (11.45,-0.4) {$\gamma_m$};

\draw [fill] (4.7,1) circle (0.7mm);
\node at (4.9,0.6) {$\gamma_1|_i$};
\draw [fill] (5.2,3) circle (0.7mm);
\node at (5.1,3.4) {$\gamma_1|_j$};

\draw [fill] (6.4,1.3) circle (0.7mm);
\node at (6.6,0.9) {$\gamma_2|_i$};
\draw [fill] (6.9,2.7) circle (0.7mm);
\node at (6.8,3.1) {$\gamma_2|_j$};

\draw [fill] (9.5,1.6) circle (0.7mm);
\node at (9.7,1.2) {$\gamma_{m-1}|_i$};
\draw [fill] (10,2.4) circle (0.7mm);
\node at (9.9,2.8) {$\gamma_{m-1}|_j$};

\draw [fill] (11.3,1.9) circle (0.7mm);
\node at (11.5,1.5) {$\gamma_m|_i$};
\draw [fill] (11.6,2.1) circle (0.7mm);
\node at (11.7,2.6) {$\gamma_m|_j$};

\draw [fill] (1.2,2.25) circle (0.7mm);
\node at (0.8,2.25) {$p_m^\prime$};
\draw [dashed] (1.2,2.25)--(11.4,2.25);

\node at (8.2,1.8) {$\dots$};

\end{tikzpicture}
\caption{Inserting  $p_m^\prime$ into $\sigma_m^\prime$ to construct $\sigma_m^{\prime\prime}$.}
\label{fig:insert}
\end{center}
\end{figure}

Having seen that splittable pairs are an obstacle to wqo, we now start working towards the converse: no splittable pairs (and no ambiguous cycles) imply wqo.
In the following sequence of lemmas (\ref{la:evex}--\ref{la:corwqo})
we will assume that $\mathcal{B}$ is a bicycle with no ambiguous cycles and no splittable pairs, and
let $C_\mathcal{B}:=\peofpa(\Path(\mathcal{B}))$ be the associated set of permutations.
The key technical fact will be that
 a balanced permutation is uniquely determined by its core and the cycle length  (Lemma \ref{la:core}).
We begin by showing that in a balanced permutation every point is `captured' inside a core copy of an expanding pair.

\begin{lemma}
\label{la:evex}
Let $\alpha_m\dots\alpha_1\beta\gamma_1\dots\gamma_m$ be a balanced permutation.
For every entry $u$ in $\alpha_m\dots\alpha_1\beta$ there exist $t\in [\cl ]$ and an expanding  pair 
$(\gamma|_p,\gamma|_q)$
(including the possibilities $(-\infty,1)$ and $(|\gamma|,\infty)$)  such that $\gamma_t|_p<u<\gamma_t|_q$.
\end{lemma}

\begin{proof}
Let us first consider the case where $u$ is a core entry, i.e. it comes from $\alpha_\cl\dots\alpha_1\beta$.
For each $t\in [\cl ]$, there exists a consecutive pair $(\gamma|_{p_t},\gamma|_{q_t})$ such that
$u$ is sandwiched by the copy of the pair in $\gamma_t$,
i.e. $\gamma_t|_{p_t}<u<\gamma_t|_{q_t}$.
We split the considerations into three cases.

\textit{Case 1: $(\gamma|_{p_t},\gamma|_{q_t})=(-\infty,1)$ for some $t$.}
Recall that the pair $(-\infty,1)$ can only be shrinking or expanding.
In this instance it cannot actually be shrinking, for this would imply
$u<\gamma_t|_{q_t}\leq \gamma_1|_{q_t}$; since $u$ belongs to the core, this would make
$(\gamma|_{p_t},\gamma|_{q_t})$ into a splittable pair.
Hence $(\gamma|_{p_t},\gamma|_{q_t})$ is expanding, and the lemma is proved in this case.

\textit{Case 2: $(\gamma|_{p_t},\gamma|_{q_t})=(|\gamma|,\infty)$ for some $t$.}
This is dual to Case 1.

\textit{Case 3: $(\gamma|_{p_t},\gamma|_{q_t})\neq (-\infty,1),(|\alpha|,\infty)$ for all $t\in [\cl]$.}
Recall that by definition \eqref{eq:ccl} we have $\cl > |\gamma|-1$, and hence
by the pigeonhole principle there exist $t_1<t_2$ such that
$(p_{t_1},q_{t_1})=(p_{t_2},q_{t_2})$; call this pair $(p,q)$.
Since $u$ is sandwiched between
$\gamma_{t_1}|_p$ and $\gamma_{t_1}|_q$, and also between
$\gamma_{t_2}|_p$ and $\gamma_{t_2}|_q$,
it follows that the pair $(\gamma|_p,\gamma|_q)$
cannot be increasing or decreasing.
Yet again, it cannot be shrinking either, for this would imply
\[
\gamma_1|_p\leq \gamma_{t_1}|_p<u<\gamma_{t_1}|_q\leq \gamma_1|_q,
\]
and the pair would be splittable.
Therefore, this pair is expanding, and the proof is complete in the case where $u$ belongs to the core.

Suppose now that $u=\alpha_r|_l$ for some $\cl<r\leq m$ and some $l\in [|\alpha|]$.
Without loss of generality assume that $\alpha_l$ is a decreasing entry, so that, in particular
\begin{equation}
\label{eq:cap1}
u=\alpha_r|_l<\alpha_\cl|_l.
\end{equation}
By what we proved in the first part there is an expanding pair $(\gamma|_p,\gamma|_q)$ such that
\begin{equation}
\label{eq:cap2}
\gamma_\cl|_p<\alpha_\cl|_l<\gamma_\cl|_q.
\end{equation}
And by a dual of the first part, there is an expanding pair
$(\alpha_g,\alpha_h)$ such that 
\begin{equation}
\label{eq:cap3}
\alpha_\cl|_g<\gamma_\cl|_p<\alpha_\cl|_h.
\end{equation}
In particular, $\alpha_h$ is an increasing entry, and so
\begin{equation}
\label{eq:cap4}
\alpha_\cl|_h<\alpha_r|_h.
\end{equation}
Also, since $\alpha_g$ and $\alpha_h$ are consecutive by value in $\alpha$, and since
$\alpha_\cl|_l>\gamma_\cl|_p$, it follows that $\alpha_\cl|_l\geq \alpha_\cl|_h$; but then
$\perm(\alpha_\cl)=\perm(\alpha_r)$ yields
\begin{equation}
\label{eq:cap5}
\alpha_r|_l\geq \alpha_r|_h.
\end{equation}
Combining \eqref{eq:cap1}--\eqref{eq:cap5} we obtain
\[
\gamma_\cl|_p<\alpha_\cl|_h<\alpha_r|_h\leq \alpha_r|_l=u<\alpha_\cl|_l<\gamma_\cl|_q,
\]
completing the proof of the lemma.
\end{proof}

Next we use the previous lemma to show that the relative position between 
an entry $u$ from $\alpha_m\dots\alpha_1\beta$ and an iterator entry from $\gamma_1\dots \gamma_m$ is the same as the relative position between $u$ and a suitable core entry.

\begin{lemma}
\label{la:primea}
Let $\alpha_m\dots \alpha_1\beta\gamma_1\dots \gamma_m$ be a balanced permutation.
For every point $u$ from $\alpha_m\dots\alpha_1\beta$, every entry $\gamma|_i$ of $\gamma$, and every
$r\in [m]$, we have:
\[
u< \gamma_r|_i\Leftrightarrow u<\gamma_{\overline{r}}|_i \quad \text{where } \overline{r}:=\min(r,\cl ).
\]
\end{lemma}

\begin{proof}
The statement is vacuous if $r\leq \cl $, so suppose $r>\cl $.
Using Lemma \ref{la:evex}, let us fix $t\in [\cl ]$ and an expanding pair $(\gamma|_p,\gamma|_q)$ 
so that $\gamma_t|_p<u<\gamma_t|_q$.
Thus no entries from $\gamma_{\cl+1}\dots \gamma_m$ are sandwiched between 
$\gamma_\cl |_p$ and $\gamma_\cl |_q$.
Therefore we have
\[
u< \gamma_r|_i \Leftrightarrow \gamma_r|_q\leq\gamma_r|_i \Leftrightarrow
\gamma_\cl |_q\leq\gamma_\cl |_i \Leftrightarrow
u< \gamma_{\overline{r}}|_i,
\]
as required.
\end{proof}

The dual statement with $\alpha$ and $\gamma$ interchanged is:

\begin{lemma}
\label{la:primeb}
Let $\alpha_m\dots \alpha_1\beta\gamma_1\dots \gamma_m$ be a balanced permutation.
For every point $u$ from $\beta\gamma_1\dots\gamma_m$, every entry $\alpha|_i$ of $\alpha$, and every
$r\in [m]$, we have:
\[
\pushQED{\qed} 
u< \alpha_r|_i\Leftrightarrow u<\alpha_{\overline{r}}|_i \quad \text{where } \overline{r}:=\min(r,\cl ). \qedhere
\popQED
\]
\end{lemma}

We consolidate this, and show that the comparison of any two entries from different `parts' of a balanced permutation can be reduced to the comparison between some corresponding core entries:

\begin{lemma}
\label{la:pull}
Let $\alpha_m\dots \alpha_1\beta\gamma_1\dots \gamma_m$ be a balanced permutation.
Then for 
all $i\in [|\alpha|]$, $j\in [|\beta|]$, $l\in [|\gamma|]$,
$r,s\in [m]$, we have
\begin{thmenumerate}
\item
\label{it:pu1}
$\alpha_r|_i<\beta|_j\Leftrightarrow \alpha_{\overline{r}}|_i<\beta|_j$, where $\overline{r}:= \min(r,\cl )$;
\item
\label{it:pu2}
$\beta|_j<\gamma_s|_l\Leftrightarrow \beta|_j<\gamma_{\overline{s}}|_l$, where
$\overline{s}:=\min(s,\cl )$;
\item
\label{it:pu3}
$\alpha_r|_i < \gamma_s|_j\Leftrightarrow
\alpha_{\overline{r}}|_i < \gamma_{\overline{s}}|_j$, where
$\overline{r}:= \min(r,\cl )$, and  $\overline{s}:=\min(s,\cl )$.
\end{thmenumerate}
\end{lemma}

\begin{proof}
Part \ref{it:pu1} follows from Lemma \ref{la:primeb}, and part \ref{it:pu2} from Lemma \ref{la:primea}.
For \ref{it:pu3}, we can use both lemmas to obtain
\[
\alpha_r|_i<\gamma_s|_j \Leftrightarrow \alpha_{\overline{r}}|_i<\gamma_s|_j
\Leftrightarrow \alpha_{\overline{r}}|_i<\gamma_{\overline{s}}|_j,
\]
as required.
\end{proof}

We can now
prove the promised result about a balanced permutation being determined by its cycle length and its core:

\begin{lemma}
\label{la:core}
For any two balanced permutations of cycle length $m$:
\[
\sigma=\alpha_m\dots\alpha_1\beta\gamma_1\dots\gamma_m
\quad\text{and}\quad
\overline{\sigma}=\overline{\alpha}_m\dots\overline{\alpha}_1\overline{\beta}\overline{\gamma}_1\dots\overline{\gamma}_m,
\]
we have
\[
\sigma=\overline{\sigma}
\quad\Leftrightarrow\quad
\perm(\alpha_\cl \dots\alpha_1\beta\gamma_1\dots\gamma_\cl )=
\perm(\overline{\alpha}_\cl \dots\overline{\alpha}_1\overline{\beta}\overline{\gamma}_1\dots\overline{\gamma}_\cl ).
\]
\end{lemma}

\begin{proof}
We show that how two points of $\alpha^{(m)}\dots\alpha^{(1)}\beta\gamma^{(1)}\dots\gamma^{(m)}$ compare to each other is uniquely determined by the core.
The comparison between two points from $\gamma_1\dots\gamma_m$
is uniquely determined by the core as a consequence of Lemmas \ref{la:dist} and \ref{la:itcomp}.
By duality, the comparison between two points from $\alpha_m\dots\alpha_1$ is uniquely determined by the core.
The comparison between a point from $\gamma_1\dots\gamma_m$ and a point from $\alpha_m\dots\alpha_1\beta$
is uniquely determined by the core via Lemma \ref{la:pull}.
The same lemma leads to the same conclusion for a point from $\beta$ and a point from $\alpha_m\dots\alpha_1$.
Finally, the comparison between two entries of $\beta$ is directly determined by the core, as $\beta$ forms a part of it.
\end{proof}

We now want to show that the set of all factors of all balanced permutations sharing the same core is wqo.
To this end we develop one more technical device.

So let $\kappa=\alpha_\cl \dots\alpha_1\beta\gamma_1\dots\gamma_\cl  $ be a core, and let $D_\kappa$ be the set of all balanced permutations from $C_\mathcal{B}$ with this core. 
We define a two-sided infinite sequence $\kappa^\ast$ of real numbers inductively as follows.
We start with $\kappa_0=\kappa$, and let $K_0$ be the set of all its entries, i.e. 
$K_0=\bigl[\cl |\alpha|+|\beta|+\cl |\gamma|\bigr]$.
Let $\alpha_{\cl +1}$, $\gamma_{\cl +1} $ be any sequences of pairwise distinct numbers from 
$\mathbb{R}\setminus K_0$ such that
\[
\perm(\alpha_{\cl +1}\alpha_\cl )=\perm(\alpha_\cl \alpha_{\cl -1})\quad\text{and}\quad
\perm(\gamma_\cl \gamma_{\cl +1})=\perm(\gamma_{\cl -1}\gamma_\cl ).
\]
Define $\kappa_1:= \alpha_{\cl +1}\kappa_0\gamma_{\cl +1}$.
Note that $\perm(\kappa_1)$ belongs to $C_\mathcal{B}$ and is balanced with cycle length $\cl +1$.
Hence $\perm(\kappa_1)$ is the unique element of $D_\kappa$ with this cycle length.

Generally, suppose we have defined a sequence $\kappa_i$ such that
$\perm(\kappa_i)$ is the unique member of $D_\kappa$ with cycle length $i$, and let $K_i$ be the set of its entries.
Let $\alpha_{i+1}$, $\gamma_{i+1}$ be two sequences using pairwise distinct numbers from
$\mathbb{R}\setminus K_i$ such that
\begin{multline*}
\perm(\alpha_{\cl +i+1}\alpha_{\cl +i})=\perm(\alpha_{\cl +i}\alpha_{\cl +i-1})
\quad\text{and}\\
\perm(\gamma_{\cl +i}\gamma_{\cl +i+1})=\perm(\gamma_{\cl +i-1}\gamma_{\cl +i}).
\end{multline*}
Define $\kappa_{i+1}:= \alpha_{\cl +i+1}\kappa_i\gamma_{\cl +i+1}$.
Note that $\perm(\kappa_{i+1})$ belongs to $C_\mathcal{B}$ and is balanced with cycle length $\cl +i+1$.
Hence $\perm(\kappa_{i+1})$ is the unique element of $D_\kappa$ with this cycle length.
Finally, we let $\kappa^\ast$ be the two-sided infinite sequence obtained as the limit of this process.
Without loss of generality we may assume that the initial copy of $\kappa$ inside $\kappa^\ast$ is
$\kappa=\kappa_0=\kappa^\ast|_{[1,l]}$, where $l:=|\kappa|=\cl |\alpha |+|\beta|+\cl |\gamma|$.

From the foregoing analysis it is clear that for each member of $D_\kappa$ there exists a subsequence of $\kappa^\ast$ that defines the same permutation.
Furthermore, every contiguous subsequence $\sigma$ of $\kappa^\ast$ is a subsequence of some $\kappa_i$, and hence defines a permutation that is a factor of a member of $D_\kappa$.

For a (finite or infinite) sequence $\sigma$ of distinct symbols, let $\Fact(\sigma)$ denote the set of all permutations corresponding to finite contiguous subsequences of $\sigma$.
Then from the above it follows that
\[
\bigcup_{\sigma\in D_\kappa} \Fact(\sigma)=\Fact(\kappa^\ast).
\]

Moreover, any permutation belonging to some $\Fact(\sigma)$ is equal to $\perm(\rho)$ where $\rho$ is a 
subsequence of $\kappa^\ast$ which contains at least one entry from $\kappa_0$.
This is immediate if $\paofpe(\sigma)$ traverses at least one edge from the connecting path of $\mathcal{B}$.
Otherwise $\paofpe(\sigma)$ is entirely contained within an iterated juxtaposition of the initial or terminal cycle
with itself.
But in that case we can choose a representative for $\sigma$ to end/start in the first augmented initial/terminal cycle.

From $\kappa_0=\kappa^\ast |_{[1,l]}$ we have
\[
\bigcup_{\sigma\in D_\kappa} \Fact(\sigma)=\bigcup \bigl\{ \perm(\kappa^\ast |_{[i,j]})\::\: i\leq l,\ j\geq 1\bigr\}.
\]
Furthermore,
\[
i\geq i_1\ \&\ j\leq j_1\Rightarrow [i,j]\subseteq [i_1,j_1]\Rightarrow
\perm(\kappa^\ast |_{[i,j]})\leq_f \perm(\kappa^\ast |_{[i_1,j_1]}).
\]
The set $(-\infty,l]\times [1,\infty)$ with the ordering
\[
(i,j)\leq (i_1,j_1)\Leftrightarrow i\geq i_1\ \&\ j\leq j_1
\]
is isomorphic to $\mathbb{N}\times\mathbb{N}$ under the natural component-wise ordering, and this latter set is wqo.
We have proved:

\begin{lemma}
\label{la:corwqo}
For any core $\kappa$, the set $D_\kappa$ of all factors of all balanced permutations with core $\kappa$ is wqo. \qed
\end{lemma}

Putting all this together we have:

\begin{lemma}
\label{la:wqob}
Let $\mathcal{B}$ be an induced bicycle of $\Gamma_\bl(S)$. The collection of permutations 
$\peofpa(\Path(\mathcal{B}))$ is wqo if and only if the following are satisfied:
\begin{thmenumerate}
\item
$\mathcal{B}$ has no ambiguous cycles; and
\item
$\mathcal{B}$ has no splittable pairs.
\end{thmenumerate}
\end{lemma}

\begin{proof}
($\Rightarrow$)
Lemmas \ref{la:wqonec}, \ref{la:splitac}.

($\Leftarrow$)
By Lemma \ref{la:corwqo}, the set of all factors of permutations from a single $D_\kappa$ is wqo.
Since there are only finitely many cores, and every permutation from $\peofpa(\Path(\mathcal{B}))$
is a factor of a balanced permutation, it follows that $\peofpa(\Path(\mathcal{B}))$ is a finite union of wqo sets, and hence is itself wqo.
\end{proof}

Finally we can prove the main theorem of this section:

\begin{thm}
\label{thm:permwqo}
Let $C$ be a finitely based downward closed set of permutations under the consecutive subpermutation ordering,
and let $\bl$ be the maximum length of a basis element of $C$.
Then $C$ is well quasi-ordered if and only if the following are satisfied:
\begin{thmenumerate}
\item
the factor graph $\Gamma(C)$ has a path-complete decomposition into bicycles $\mathcal{B}_1$, \dots, $\mathcal{B}_n$;
\item
$\Gamma(C)$ has no ambiguous cycles; and
\item
no bicycle $\mathcal{B}_i$  has a splittable pair.
\end{thmenumerate}
\end{thm}

\begin{proof}
($\Rightarrow$)
Lemmas \ref{la:wqonec} and \ref{la:wqob}.

($\Leftarrow$)
Recalling \eqref{eq:fu}, we have
\[
C=C_{[\bl)}\cup\bigcup_{i=1}^n \peofpa(\Path(\mathcal{B}_i)).
\]
Each $\peofpa(\Path(\mathcal{B}_i))$ is wqo by Lemma \ref{la:wqob},
and $C_{[\bl)}$ is wqo because it is finite. Hence $C$ is a finite union of wqo sets, and is therefore itself wqo.
\end{proof}

And we can prove our final main result of the paper as an immediate corollary:

\begin{proof}[Proof of Theorem \ref{thm:perms} \ref{it:perms-wqo}]
The graph $\Gamma(C)$ can be effectively computed from $\beta_1,\dots,\beta_n$.
Then it can be decided whether it has a path-complete decomposition into bicycles, or, equivalently, whether it has an in-out cycle; see Theorem \ref{thm:grwqo}.
Whether there are any ambiguous cycles is decidable by computing all minimal ambiguous paths (cf. Lemmas \ref{lem:bound}, \ref{lem:ambdec}), and checking whether any of them are subpaths of a cycle.
Finally, whether there are any splittable shrinking pairs in any of the bicycles of $\Gamma(C)$ is decidable by Lemma \ref{la:decsplit}.
\end{proof}

\textbf{Acknowledgement.}
The authors would like to thank the two referees for their careful reading of the paper and 
suggestions that have improved it.

\bibliographystyle{plain}

\end{document}